\documentclass[12pt, a4paper]{article}
\usepackage{amsmath,amsthm,amssymb,amsfonts,setspace,enumerate,wrapfig,graphicx}
\graphicspath{{/homes/homes0/v1/pg_2007/white/Documents/Pictures/}}
\newtheorem{thm}{Theorem}
\newtheorem*{defn}{Definition}

\newtheorem{question}[thm]{Question}
\newtheorem{corollary}[thm]{Corollary}
\newtheorem{conj}[thm]{Conjecture}
\newtheorem{claim}[thm]{Claim}
\newtheorem{prob}[thm]{Problem}
\newtheorem{lemma}[thm]{Lemma}

\begin{document}
\title{Monochromatic cycles and the monochromatic circumference in $2$-coloured graphs}
\author{Alex Scott and Matthew White \\
Mathematical Institute \\ 
24-29 St Giles' \\
Oxford \\
OX1 3LB \\
England 
\thanks{email: scott@maths.ox.ac.uk; white@maths.ox.ac.uk}}
\maketitle

\begin{abstract}
Li, Nikiforov and Schelp \cite{Li2010} conjectured that a $2$-edge coloured graph $G$ with order $n$ and $ \delta(G) >\frac{3}{4} n$ contains a monochromatic cycle of length $\ell$, for all $\ell \in [4,\left\lceil \frac{n}{2} \right\rceil]$. We prove this conjecture for sufficiently large $n$ and also find all $2$-edge coloured graphs with $\delta(G)=\frac{3}{4}n$ that do not contain all such cycles. Finally we show that, for all $\delta>0$ and $n>n_{0}(\delta)$, a $2$-edge coloured graph $G$ of order $n$ with $\delta(G) \geq\frac{3}{4} n$ either contains a monochromatic cycle of length at least $\left(\frac{2}{3}+\frac{\delta}{2}\right)n$, or contains a monochromatic cycle of length $\ell$, in the same colour, for all $\ell \in [3,\left(\frac{2}{3}-\delta\right)n]$.
\end{abstract}

\section{Introduction}
\label{Sec: Ramsey intro}
A well-known theorem of Dirac \cite{Dirac1952} states that a graph with order $n\geq 3$ and minimal degree at least $\frac{1}{2}n$ contains a Hamilton cycle.
\begin{thm}[Dirac \cite{Dirac1952}]
\label{Thm: Dirac}
Let $G$ be a graph of order $n\geq 3$. If $\delta(G) \geq \frac{1}{2}n$, then $G$ is hamiltonian.
\end{thm}
In fact, as noted by Bondy \cite{Bondy1971}, an immediate corollary of the following theorem is that such a graph will contain cycles of all lengths $\ell \in [3,n]$. We call such a graph \emph{pancyclic}.
\begin{thm}[Bondy \cite{Bondy1971}]
\label{Thm: bondy}
If $G$ is a hamiltonian graph of order $n$ such that $|E(G)|\geq \frac{n^{2}}{4}$, then either $G$ is pancyclic or $n$ is even and $G\cong K_{n/2,n/2}$.
\end{thm}
\begin{corollary}
\label{Cor: bondy and dirac}
Let $G$ be a graph of order $n\geq 3$. If $\delta(G) \geq \frac{1}{2}n$, then either $G$ is pancyclic or $n$ is even and $G\cong K_{n/2,n/2}$.
\end{corollary}

Given a graph $G$ with edge set $E(G)$, a $2$-edge colouring of $G$ is a partition $E(G) = E(R) \cup E(B)$, where $R$ and $B$ are spanning subgraphs of $G$. In a recent paper \cite{Li2010}, Li, Nikiforov and Schelp made the following conjecture, which will give an analogue of Corollary \ref{Cor: bondy and dirac} for $2$-edge coloured graphs.

\begin{conj}
Let $n\geq 4$ and let $G$ be a graph of order $n$ with $\delta\left(G\right) > \frac{3}{4}n$. If $E(G)=E(R) \cup E(B)$ is a $2$-edge colouring, then $C_{\ell} \subseteq R$ or $C_{\ell} \subseteq B$ for all $\ell \in \left[4,\left\lceil \frac{1}{2}n\right\rceil \right]$.
\end{conj}

Note that we may only ask for $\ell$ in this range. For example, taking the $2$-colouring of $K_{5}$ as a red and blue $C_{5}$ and blowing up, we get a graph with $\delta(G)=\frac{4}{5}|G|$ but no monochromatic $C_{3}$. Similarly letting $R$ be the complete bipartite graph with vertex classes of order $\lfloor \frac{n}{2} \rfloor$ and $\lceil \frac{n}{2} \rceil$, and letting $B$ be the complement, we obtain a $2$-colouring of the complete graph with no monochromatic odd cycle of length $\ell > \lceil \frac{n}{2} \rceil$. In \cite{Li2010}, Li, Nikiforov and Schelp proved the following partial result.

\begin{thm}
\label{Thm: li's result}
Let $\epsilon >0$, let $G$ be a graph of sufficiently large order $n$, with $\delta\left(G\right) > \frac{3}{4}n$. If $E(G)=E(R) \cup E(B)$ is a $2$-edge colouring, then $C_{\ell} \subseteq R$ or $C_{\ell} \subseteq B$ for all $\ell \in \left[4,\left\lfloor \left(\frac{1}{8}-\epsilon\right)n\right\rfloor \right]$.
\end{thm}

We will prove the conjecture for sufficiently large $n$, but first we will define a set of $2$-edge coloured graphs showing that the degree bound $\frac{3}{4}n$ is tight.

\begin{defn}
Let $n=4p$ and let $G$ be isomorphic to $K_{p,p,p,p}$. A \emph{$2$-bipartite $2$-edge colouring} of $G$ is a $2$-edge colouring $E(G)=E(R) \cup E(B)$ such that both $R$ and $B$ are bipartite.
\end{defn}

If $G \cong K_{p,p,p,p}$ and $G$ has a $2$-bipartite $2$-edge colouring, let $V_{1} \cup V_{2}$ be the bipartition of $R$ and $W_{1} \cup W_{2}$ be the bipartition of $B$. Let $U_{i,j}=V_{i} \cap W_{j}$ for all $i,j \in \{1,2\}$. Then the $U_{i,j}$ are four independent sets of $G$ covering the vertices, and so must be the four independent sets of order $p$. So, a $2$-bipartite $2$-edge colouring of $K_{p,p,p,p}$ forces a labelling of the independent sets $\left\{U_{i,j}:i,j \in \{1,2\}\right\}$ such that:
\begin{itemize}
\item all edges between $U_{1,1}$ and $U_{1,2}$ and between $U_{2,1}$ and $U_{2,2}$ are blue;
\item all edges between $U_{1,1}$ and $U_{2,1}$ and between $U_{1,2}$ and $U_{2,2}$ are red;
\item edges between $U_{1,1}$ and $U_{2,2}$ and between $U_{2,1}$ and $U_{1,2}$ can be either colour.
\end{itemize}
Hence, if $4$ divides $n$, the graph $K_{n/4,n/4,n/4,n/4}$ with a $2$-bipartite $2$-edge colouring has minimal degree $\frac{3}{4}n$ and no monochromatic odd cycles. Note that for a fixed labelling of the graph, there are $2^{2p^{2}}$ $2$-bipartite $2$-edge colourings of $K_{p,p,p,p}$. However, $K_{p,p,p,p}$ has $24\left(p!\right)^{4}=2^{O\left(p\log p\right)}$ automorphisms and so there are $2^{2p^{2}+O\left(p\log p\right)}$ distinct $2$-bipartite $2$-edge colourings of $K_{p,p,p,p}$. In fact we will prove that $K_{n/4,n/4,n/4,n/4}$ is the only extremal graph; although any $2$-bipartite $2$-edge colouring of $K_{n/4,n/4,n/4,n/4}$ is extremal.

\begin{thm}
\label{Thm: main result}
Let $n$ be sufficiently large and let $G$ be a graph of order $n$ with $\delta\left(G\right) \geq \frac{3}{4}n$. Suppose that $E(G)=E(R_{G}) \cup E(B_{G})$ is a $2$-edge colouring. Then either $C_{\ell} \subseteq R$ or $C_{\ell} \subseteq B$ for all $k \in \left[4,\left\lceil \frac{1}{2}n\right\rceil \right]$, or $n=4p$, $G \cong K_{p,p,p,p}$ and the colouring is a $2$-bipartite $2$-edge colouring.
\end{thm}

We define the \emph{monochromatic circumference} of a $k$-edge coloured graph $G$ to be the length of the longest monochromatic cycle. In \cite{Li2010}, the authors also posed the following question.

\begin{question}
Let $0<c<1$ and $G$ be a graph of sufficiently large order $n$. If $\delta(G)>cn$ and $E(G)$ is $2$-coloured, how long monochromatic cycles are there?
\end{question}

For graphs $G$ with $\delta(G) \geq \frac{3}{4}n$ we show that the monochromatic circumference is at least $\left(1+o(1)\right) \frac{2}{3}n$. In fact, we show the following result.

\begin{thm}
\label{Thm: circumference}
Let $n$ be sufficiently large and $0<\delta\leq \frac{1}{180}$. Let $G$ be a graph of order $n$ with $\delta(G) \geq \frac{3}{4}n$. Suppose that $E(G)=E(R_{G}) \cup E(B_{G})$ is a $2$-edge colouring. Then either $G$ has monochromatic circumference at least $\left(\frac{2}{3}+\frac{\delta}{2} \right)n$, or one of $R_{G}$ and $B_{G}$ contains $C_{\ell}$ for all $\ell \in\left[3,\left(\frac{2}{3}-\delta\right)n\right]$.
\end{thm}

Note that the last statement requires monochromatic cycles of all lengths in some prescribed set of integers, as in Theorem \ref{Thm: main result}. However, here these cycles are required to be of the \emph{same} colour. Also, the upper bound on $\delta$ is of a technical nature, and we are interested in small $\delta$. There are similar technical bounds throughout.

For integers $t\leq s$, we define the following $2$-edge coloured graph, which with $s=2t$ shows that Theorem \ref{Thm: circumference} is asymptotically sharp.

\begin{defn}
Let $F_{s,t}$ be the complete graph on $t+s$ vertices, with $V=V(F_{s,t})$. Let $A\subseteq V$ be a set of order $s$. We $2$-edge colour $F_{s,t}$ by letting all edges between $A$ and $V\setminus A$ be blue, and all other edges be red. The blue graph is bipartite and has circumference $2|V\setminus A| =2t$. The red graph has circumference $s$. Thus the monochromatic circumference of $F_{s,t}$ is $\max\{s,2t\}$.
\end{defn}

Let $n=3t$. Then $|F_{2t,t}|=n$, $\delta(F_{2t,t})=n-1$ and $F_{2t,t}$ has monochromatic circumference $\frac{2}{3}n$. Hence Theorem \ref{Thm: circumference} is asymptotically sharp.

We shall show that a linear dependence between the two occurences of $\delta$ in Theorem \ref{Thm: circumference} is correct. Fix $\delta > 0$. Let $G \cong F_{n-\left\lceil\left(\frac{2}{3}-\delta\right)n\right\rceil,\left\lceil\left(\frac{2}{3}-\delta\right)n\right\rceil}$. Then the monochromatic circumference of $G$ is at most $\left(\frac{2}{3}+2\delta\right)n$. However, $G$ contains no monochromatic cycle of length $\ell$ where $\ell$ is whichever of $\left\{\left\lceil\left(\frac{2}{3}-\delta\right)n\right\rceil+1,\left\lceil\left(\frac{2}{3}-\delta\right)n\right\rceil+2\right\}$ is odd.

In Section \ref{Sec: tools}, we will introduce some theorems that will be used in our proofs. We will then prove Theorem \ref{Thm: main result} in two parts. Section \ref{Sec: short odd cycles} will deal with short (up to constant length) cycles and Section \ref{Sec: regularity section} will deal with long cycles. This will rely on a number of lemmas, which are proved in Section \ref{Sec: proof of extremal lemma}. In Section \ref{Sec: Monochromatic circumference} we will look at the length of the longest monochromatic cycle, and in particular prove Theorem \ref{Thm: circumference}. We conclude in Section \ref{Sec: conclusion} with some open problems.

\section{Results used in the proof}
\label{Sec: tools}
In order to prove Theorem \ref{Thm: main result}, we shall use the common extremal graph theory method of the Regularity Lemma and Blow-up Lemmas to find long cycles. Before introducing these, we make some preliminary definitions.
\begin{defn}
Let $G$ be a graph and $X$ and $Y$ be disjoint subsets of $V(G)$. The \emph{density} of the graph $G[X,Y]$ is the value
\begin{displaymath}
d(X,Y):= \frac{e(X,Y)}{|X||Y|}.
\end{displaymath}
\end{defn}
We define a regular pair to be one where the density between not-too-small subgraphs of $X$ and $Y$ is close to the density between $X$ and $Y$.
\begin{defn}[Regularity]
Let $\epsilon>0$. Let $G$ be a graph and $X$ and $Y$ be disjoint subsets of $V(G)$. We call $(X,Y)$ an \emph{$\epsilon$-regular pair for $G$} if, for all $X' \subseteq X$ and $Y' \subseteq Y$ satisfying $|X'|\geq \epsilon |X|$ and $|Y'| \geq \epsilon |Y|$, we have
\begin{displaymath}
\left| d(X,Y) -d(X',Y') \right| < \epsilon.
\end{displaymath}
\end{defn}
It is often useful to have a bound on the degree of vertices in $X$ and $Y$.
\begin{defn}[Super-regularity]
Let $\epsilon,\delta>0$. Let $G$ be a graph and $X$ and $Y$ be disjoint subsets of $V(G)$. We call $(X,Y)$ an \emph{$(\epsilon,\delta)$-super-regular pair for $G$} if, for all $X' \subseteq X$ and $Y' \subseteq Y$ satisfying $|X'|\geq \epsilon |X|$ and $|Y'| \geq \epsilon |Y|$,
\begin{displaymath}
e(X',Y') > \delta|X'||Y'|,
\end{displaymath}
and furthermore, $d_{Y}(v) > \delta |Y|$ for all $v \in X$ and $d_{X}(v) > \delta |X|$ for all $v \in Y$.
\end{defn}
Note that a super-regular pair need not be regular, as the number of edges between subsets is only bounded below.

We will use the following $2$-coloured version of the Szemer\'{e}di Regularity Lemma \cite{Szemeredi1976} (see, for example, the survey paper of Koml\'{o}s and Simonovits \cite{Komlos1996} for an edge-coloured version).
\begin{thm}[Degree form of $2$-coloured Regularity Lemma]
\label{Thm: regularity}
For every $\epsilon >0$ there is an $M=M\left(\epsilon\right)$ such that if $G=(V,E)$ is any $2$-coloured graph and $d \in [0,1]$ is any real number, then there is $k \leq M$, a partition of the vertex set $V$ into $k+1$ clusters $V_{0},V_{1},\ldots,V_{k}$, and a subgraph $G'\subseteq G$ with the following properties:
\begin{itemize}
\item $|V_{0}| \leq \epsilon |V|$,
\item all clusters $V_{i}$, $i\geq 1$, are of the same size $m\leq \left\lceil\epsilon|V|\right\rceil$,
\item $d_{G'}\left(v\right) > d_{G}\left(v\right) - \left(2d+\epsilon\right) |V|$ for all $v \in V$,
\item $e\left(G'\left(V_{i}\right)\right)=0$ for all $i\geq 1$,
\item for all $1\leq i< j \leq k$, the pair $\left(V_{i},V_{j}\right)$ is $\epsilon$-regular for $R_{G'}$ with a density either $0$ or greater than $d$ and $\epsilon$-regular for $B_{G'}$ with a density either $0$ or greater than $d$, where $E(G')=E(R_{G'}) \cup E(B_{G'})$ is the induced $2$-edge colouring of $G'$.
\end{itemize}
\end{thm}

Having applied the above form of the Regularity Lemma to a $2$-coloured graph $G$, we make the following definition, based on the clusters $\{V_{i}:1\leq i \leq k \}$. Note that this definition depends on the parameters $\epsilon$ and $d$.
\begin{defn}[Reduced graph]
We define a \emph{$(\epsilon,d)$-reduced $2$-edge coloured graph $H$} on vertex set $\{v_{i}:1\leq i \leq k\}$ as follows:
\begin{itemize}
\item let $\{v_{i},v_{j}\}$ be a blue edge of $H$ when $B_{G'}\left[V_{i},V_{j}\right]$ has density at least $d$;
\item let $\{v_{i},v_{j}\}$ be a red edge of $H$ when it is not a blue edge and $R_{G'}\left[V_{i},V_{j}\right]$ has density at least $d$.
\end{itemize}
\end{defn}
We aim to use subgraphs of the reduced graph $H$ to find subgraphs of $G$. To do so we will use the Embedding Lemma and the Blow-up Lemma of Koml\'{o}s, S\'{a}rk\"{o}zy and Szemer{\'{e}}di \cite{Komlos1997}.
\begin{thm}[Embedding Lemma]
\label{Thm: embedding lemma}
Given $d>\epsilon >0$, a graph $H$, and a positive integer $m$, let us construct a graph $G$ by replacing each vertex of $H$ with a set of order $m$, and replacing the edges of $H$ with $\epsilon$-regular pairs of density at least $d$. For a fixed integer $t$, let $H(t)$ be the graph defined by replacing each vertex of $H$ with a set of order $t$, and replacing the edges of $H$ with the complete bipartite graph.

Let $F$ be a subgraph of $H(t)$ with $f$ vertices and maximum degree $\Delta >0$, and let $\eta=d-\epsilon$ and $\epsilon_{0}=\eta^{\Delta}/\left(2+\Delta\right)$. If $\epsilon \leq \epsilon_{0}$ and $t-1 \leq \epsilon_{0} m$, then $F \subseteq G$, and in fact $G$ contains at least $(\epsilon_{0}m)^{f}$ vertex disjoint copies of $F$.
\end{thm}

Note that in the Embedding Lemma, the graphs $F$ we embed into $G$ have order at most $t|H|=\frac{t}{m}|G|$. We will need to embed much larger graphs into $G$; for this we will need the Blow-up Lemma. Note that here we consider \emph{super}-regular pairs.

\begin{thm}[Blow-up Lemma]
\label{Thm: blow-up lemma}
Given a graph $H$ of order $r$ and positive parameters $\delta$, $\Delta$ and $c$, there exist positive numbers $\epsilon=\epsilon\left(\delta,\Delta,r,c\right)$ and $\alpha = \alpha\left(\delta,\Delta,r,c\right)$ such that the following holds. Let $t$ be an arbitary positive integer, and replace the vertices $v_{1},\ldots ,v_{r}$ of $H$ with pairwise disjoint sets $V_{1},\ldots,V_{r}$ of order $t$. We construct two graphs on the same vertex set $V=\bigcup V_{i}$. The first graph $G_{1}$ is obtained by replacing each edge $\{v_{i},v_{j}\}$ of $H$ with the complete bipartite graph between the corresponding vertex sets $V_{i}$ and $V_{j}$. A sparser graph $G_{2}$ is constructed by replacing each edge $\{v_{i},v_{j}\}$ arbitarily with an $\left(\epsilon,\delta\right)$-super-regular pair between $V_{i}$ and $V_{j}$. If a graph $F$ with $\Delta\left(F\right) \leq \Delta$ is embeddable into $G_{1}$ then it is also embeddable into $G_{2}$. This remains true even if for every $i$ there are certain vertices $x$ to be embedded into $V_{i}$ whose images are \textit{a priori} restricted to certain sets $C_{x} \subseteq V_{i}$, provided that:
\begin{enumerate}[(i)]
\item each $C_{x}$ is of order at least $ct$;
\item the number of such restrictions within a set $V_{i}$ is not more than $\alpha t$.
\end{enumerate}
\end{thm}

In the proof of Theorem \ref{Thm: main result}, we shall frequently show that there is a subset $S$ of $V$ on which one of $R_{G}$ or $B_{G}$ is hamiltonian, and apply Theorem \ref{Thm: bondy}. To prove hamiltonicity, it will normally be sufficient to use Dirac's Theorem (Theorem \ref{Thm: Dirac}). However, we will also need the following generalisation.

\begin{thm}[Chv\'{a}tal \cite{Chvatal1972}]
\label{Thm: chvatal}
Let $G$ be a graph of order $n\geq3$ with degree sequence $d_{1} \leq d_{2} \leq \cdots \leq d_{n}$ such that
\begin{displaymath}
d_{k} \leq k < \frac{1}{2}n \Rightarrow d_{n-k} \geq n-k.
\end{displaymath}
Then $G$ contains a Hamilton cycle.
\end{thm}

In Section \ref{Sec: proof of extremal lemma}, we will need to use Tutte's $1$-factor Theorem \cite{Tutte1947}, for which we need the following definition.
\begin{defn}
For any graph $G$, let \emph{$q(G)$} be the number of components of $G$ of odd order.
\end{defn}
In fact we shall use the following defect version of Tutte's Theorem, noted by Berge \cite{Berge1958}.
\begin{thm}
\label{Thm: Tutte}
A graph $G$ contains a set of independent edges covering all but at most $d$ of the vertices if, and only if
\begin{displaymath}
 q(G-S) \leq |S| +d
\end{displaymath}
for all $S \subseteq V$.
\end{thm}

The next result of Bollob\'{a}s \cite[p.150]{bollobasextremalgraphtheory} will be used in Section \ref{Sec: short odd cycles}, as will the three following results.
\begin{thm}[Bollob\'{a}s \cite{bollobasextremalgraphtheory}]
\label{Thm: bollobas}
If $G$ is a graph of order $n$, with $e(G)>\frac{1}{4}n^{2}$, then $G$ contains $C_{k}$ for all $k \in \left[3, \left\lceil \frac{n}{2} \right\rceil \right]$.
\end{thm}

\begin{thm}[Bondy and Simonovits \cite{Bondy1974}]
\label{Thm: bondy/simonovits}
Let $G$ be a graph of order $n$ and let $k$ be an integer. If $e(G)>100kn^{1+1/k}$, then $G$ contains a cycle of length $2k$.
\end{thm}

\begin{thm}[Erd\H{o}s and Gallai \cite{Erdos1959}]
\label{Thm: erdos-gallai}
If $G$ is a graph with order $n$ and circumference at most $L$, then $e(G) \leq \frac{1}{2}\left(n-1\right)L$. If $G$ is a graph with order $n$ with no paths of length at least $L+1$, then $e(G) \leq \frac{1}{2}nL$.
\end{thm}

\begin{thm}[Gy\"{o}ri, Nikiforov and Schelp \cite{Gyori2003}]
\label{Thm: gyori et al}
Let $k,m$ be positive integers. There exist $n_{0}=n_{0}\left(k,m\right)$ and $c=c\left(k,m\right)>0$ such that for every nonbipartite $G$ on $n>n_{0}$ vertices with minimum degree
\begin{displaymath}
\delta > \frac{n}{2\left(2k+1\right)} +c,
\end{displaymath}
if $C_{2s+1} \subseteq G$, for some $k \leq s \leq 4k+1$, then $C_{2s+2j+1} \subseteq G$ for every $j \in [m]$.
\end{thm}

\section{Existence of short cycles}
\label{Sec: short odd cycles}
In this section we shall prove that unless we are in the extremal case, we have monochromatic cycles of all lengths $\ell \in [4,K]$ for a given integer $K$.

\begin{lemma}
\label{Lemma: odd cycles}
Let $K$ be an integer. Let $n$ be sufficiently large and let $G$ be a graph of order $n$ with $\delta\left(G\right) \geq \frac{3}{4}n$. If $E(G)=E(R_{G}) \cup E(B_{G})$ is a $2$-edge colouring, then either $C_{\ell} \subseteq R$ or $C_{\ell} \subseteq B$ for all $\ell \in \left[4,K\right]$, or $n=4p$, $G \cong K_{p,p,p,p}$ and the colouring is a $2$-bipartite $2$-edge colouring.
\end{lemma}

To prove this we shall use the following claim. The proof of Claim \ref{Claim: C3 or C5 is enough} follows exactly the method used in \cite{Li2010} to show the existence of short odd cycles. Note that we can not appeal directly to Theorem \ref{Thm: li's result}, as the assumption there is that $\delta(G) > \frac{3}{4}n$, whereas in Theorem \ref{Thm: main result} we assume only that $\delta(G) \geq \frac{3}{4}n$.

\begin{claim}
\label{Claim: C3 or C5 is enough}
Let $L$ be an integer. Let $n$ be sufficiently large and let $G$ be a graph of order $n$ with $\delta\left(G\right) \geq \frac{3}{4}n$. Suppose that $E(G)=E(R_{G}) \cup E(B_{G})$ is a $2$-edge colouring. If there is a monochromatic $C_{3}$ or $C_{5}$, then there is a monochromatic $C_{\ell}$ for all odd $\ell \in \left[5,2L+1\right]$.
\end{claim}
\begin{proof}
Suppose first that $\Delta\left(B\right) > \frac{1}{2}n + 4L$. Let $v$ be a vertex with $d_{B}\left(v\right)=\Delta\left(B\right)$, and $U = \Gamma_{B}\left(v\right)$. If $B[U]$ contains a path of length $2L$, then using the vertex $v$, there is a blue $C_{\ell}$ for all $\ell \in [3,2L+1]$. Hence $B[U]$ does not contain a path of length $2L$, and hence by Theorem \ref{Thm: erdos-gallai} we have $e\left(B[U]\right) \leq L |U|$. However, any vertex $v\in U$ has at most $\frac{1}{4}n$ non-neighbours in $U$ and so at least $|U|-\frac{1}{4}n$ neighbours. Hence
\begin{align*}
e\left(G[U]\right) &= \frac{1}{2}\sum_{u \in U} d_{G[U]}\left(u\right) \\
&\geq \frac{1}{2} |U| \left(|U|-\frac{1}{4}n\right) \\
&> \frac{1}{2}|U| \left(\frac{1}{2}|U| + 2L\right).
\end{align*}
Hence $e\left(R[U]\right)=e(G[U])-e(R[U]) >  \frac{1}{4}|U|^{2}$, and so by Theorem \ref{Thm: bollobas}, $R[U]$ has cycles of all lengths from $3$ to $\frac{1}{2}|U|$.

So we may assume that $\Delta\left(B\right) \leq \frac{1}{2}n + 4L$, and hence 
\begin{displaymath}
\delta\left(R\right) \geq \frac{1}{4}n - 4L > \frac{1}{6}n+c(1,L),
\end{displaymath}
where $c(1,L)$ is the constant from Theorem \ref{Thm: gyori et al}. Similarly we may assume that $\delta\left(B\right) > \frac{1}{6}n+c(1,L)$. Suppose that there is a monochromatic $C_{3}$ or $C_{5}$ and assume without loss of generality that it is red. Applying Theorem \ref{Thm: gyori et al} to $R$ with $L=1$ and $m=K$, there is a red $C_{\ell}$ for all odd $\ell \in [5,2L+1]$ as required.
\end{proof}

\begin{proof}[Proof of Lemma \ref{Lemma: odd cycles}]
Note that the existence of monochromatic $C_{\ell}$ for all even $\ell \in \left[4,K\right]$ is immediate from Theorem \ref{Thm: bondy/simonovits}. Hence, by Claim \ref{Claim: C3 or C5 is enough}, it is sufficient to prove that either there is a monochromatic $C_{3}$ or $C_{5}$, or $n=4p$, $G \cong K_{p,p,p,p}$ and the colouring is a $2$-bipartite $2$-edge colouring. Suppose that, in fact, none of these occur. Any $2$-edge colouring of $K_{5}$ contains a monochromatic $C_{3}$ or $C_{5}$. Hence we may assume that $K_{5} \nsubseteq G$. By Tur\'{a}n's Theorem, we must therefore have that $G \cong T_{4}(n)$. However, $\delta(G) \geq \frac{3}{4}n$ implies that in fact $n=4p$ and hence $G \cong K_{p,p,p,p}$. Let $U_{i}$ ($1\leq i \leq 4$) be the independent sets of $G$ of order $p$.

We may assume that $R$ is not bipartite. Let $C=v_{1}v_{2}\ldots v_{r}$ be a shortest odd cycle of $R$; we may assume that $r\geq 7$. We may properly $4$-vertex colour $C$ by setting $c(v_{i}) = j$ when $v_{i} \in U_{j}$. As $C$ is an odd cycle, there must be three consecutive vertices with different colours under $c$. Without loss of generality, assume that $c(v_{3})=1$, $c(v_{4})=2$ and $c(v_{5})=3$.

We will aim to show that $G[V(C)]$ contains a triangle or $5$-cycle which is edge-disjoint from $C$. Then we may assume that an edge of the triangle or $5$-cycle is red, else we have a monochromatic $C_{3}$ or $C_{5}$. But this red edge, together with $C$, will create a shorter red odd cycle than $C$, contradicting our assumption that $C$ was minimal. We shall find such a triangle or $5$-cycle by case analysis.

If $c(v_{1})$ is $2$ or $4$, then $G$ contains the triangle $v_{1}v_{3}v_{5}$, as these vertices lie in different $U_{j}$. Hence $c(v_{1}) \in \{1,3\}$, and similarly $c(v_{7}) \in \{1,3\}$.

If $c(v_{6})=4$, then $G$ contains the triangle $v_{1}v_{4}v_{6}$. So we may assume that $c(v_{6}) \neq 4$ and similarly $c(v_{2}) \neq 4$. Hence $c(v_{2}) \in \{2,3\}$ and $c(v_{6}) \in \{1,2\}$. If $c(v_{2})=3$ and $c(v_{6})=1$, then $G$ contains the triangle $v_{2}v_{4}v_{6}$. Hence, by symmetry, we may assume that $c(v_{2})=2$ and $c(v_{6}) \in \{1,2\}$.

If $c(v_{7})=1$, then $G$ contains the triangle $v_{2}v_{5}v_{7}$. Hence $c(v_{7})=3$.

If $|C|=7$, then as $c$ is a proper colouring, we have $c(v_{1})=1$. But then $v_{1}v_{5}v_{2}v_{7}v_{4}$ is a $5$-cycle in $G$, not containing any edges of $C$. So we may assume that $|C|>7$, and in particular $v_{1}v_{7} \notin E(C)$.

If $c(v_{1})=1$, then $v_{1}v_{4}v_{7}$ is a triangle in $G$. Hence $c(v_{1})=3$. But now, if $c(v_{6})=1$, then $G$ contains the triangle $v_{1}v_{4}v_{6}$, while if $c(v_{6})=2$, $G$ contains the triangle $v_{1}v_{3}v_{6}$, giving a contradiction.

Hence, in fact, our assumption was false, and one of the cases of the lemma holds.
\end{proof}

\section{Existence of long cycles}
\label{Sec: regularity section}
In order to find long monochromatic cycles, we will use the Regularity Lemma. Recall from Section \ref{Sec: tools} that having applied the Regularity Lemma to $G$, we define a reduced graph $H$. Note that the Regularity Lemma implies that the minimal degree of the reduced graph $H$ is not too much smaller than $\frac{k}{n}$ times the minimal degree of $G$. 

Suppose that the red edges of our reduced graph $H$ contain a large set of independent edges. Then we can use the Blow-up Lemma to create lots of long red paths in $G$, which we can hope to join together into long red cycles. One situation in which we could join together the paths in $G$ obtained from a monochromatic matching in $H$ is when the matching is contained in a component of the relevant colour in $H$. Then we can use the properties of regular pairs, and in particular the Embedding Lemma, to join the paths. The following lemma, proved in Section \ref{Sec: proof of extremal lemma} using extremal arguments, shows that if there is no monochromatic component of $H$ containing a large matching, then the reduced graph has one of two particular forms.

\begin{lemma}
\label{Lemma: reduced graph}
Let $0<\delta < \frac{1}{36} $ and let $G$ be a graph of sufficiently large order $n$ with $\delta(G) \geq \left( \frac{3}{4} -\delta \right)n$. Suppose that we are given a $2$-edge colouring $E(G)=E(R) \cup E(B)$. Then one of the following holds.
\begin{enumerate}[(i)]
\item There is a component of $R$ or $B$ which contains a matching on at least $\left( \frac{2}{3} +\delta \right)n$ vertices.
\item There is a set $S$ of order at least $\left(\frac{2}{3}-\frac{\delta}{2} \right)n$ such that either $\Delta(R[S]) \leq 10\delta n$ or $\Delta(B[S]) \leq 10\delta n$.
\item There is a partition $V(G)=U_{1} \cup \cdots \cup U_{4}$ with $\underset{i}{\min} |U_{i}| \geq \left(\frac{1}{4}-3\delta\right)n$ such that there are no red edges from $U_{1} \cup U_{2}$ to $U_{3} \cup U_{4}$ and no blue edges from $U_{1} \cup U_{3}$ to $U_{2} \cup U_{4}$.
\end{enumerate} 
\end{lemma}

In the first case, we will need the following lemma, which is also proved in Section \ref{Sec: proof of extremal lemma}.

\begin{lemma}
\label{Lemma: large matching in large component or with odd cycle}
Let $0<\delta < \frac{1}{6}$ and let $G$ be a graph of sufficiently large order $n$ with $\delta(G) \geq \left( \frac{3}{4} -\delta \right)n$. Suppose that we are given a $2$-edge colouring $E(G)=E(R) \cup E(B)$. Suppose that there is a monochromatic component containing a matching on at least $\left(\frac{2}{3}+\delta\right)n$ vertices. Then there is a monochromatic component $C$ containing a matching on at least $\left(\frac{1}{2}+\delta \right)n$ vertices such that either $C$ contains an odd cycle, or $|C|\geq \left(1-5\delta\right)n$.
\end{lemma}

By analysing the original graph, we will use the following two lemmas to show that, in the second and third cases of Lemma \ref{Lemma: reduced graph}, we will have the desired monochromatic cycles. In both of the following two lemmas, we assume the following setup. We have constants $0<\epsilon \ll d \ll \delta <\frac{1}{144}$. Let $n$ be sufficiently large and $G$ a graph of order $n$ with $\delta\left(G\right) \geq \frac{3}{4}n$. Suppose that $E(G)=E(R_{G}) \cup E(B_{G})$ is a $2$-edge colouring. We find a regular partition of $G$ using Theorem \ref{Thm: regularity}, and as defined in Section \ref{Sec: tools}, let $H$ be the $(\epsilon,d)$-reduced $2$-edge coloured graph obtained from this partition.

The following results will be useful, and will be proved in Section \ref{Sec: proof of extremal lemma}
\begin{lemma}
\label{Lemma: no large independent sets}
If $B_{G}$ has an independent set $S$ with $|S|> \frac{1}{2}n$, then $C_{\ell} \subseteq R_{G}$ for all $\ell \in [3, |S|]$. Further, if $B_{G}$ is bipartite, then either $C_{\ell} \subseteq R_{G}$ for all $\ell \in [4, \lceil\frac{n}{2}\rceil]$, or $n$ is divisible by four, $G\cong K_{n/4,n/4,n/4,n/4}$ and the colouring is a $2$-bipartite $2$-edge colouring. 
\end{lemma}

\begin{lemma}
\label{Lemma: large sparse set}
If there is a set $S\subseteq V(H)$ of order at least $\left(\frac{2}{3}-\frac{\delta}{2}\right)k$ such that $\Delta(R_{H}[S]) \leq 10\delta k$, then $G$ contains a blue cycle of length $\ell$ for all $\ell \in [3, \left(\frac{2}{3}-\delta\right)n]$.
\end{lemma}

\begin{lemma}
\label{Lemma: pathological case}
Suppose that there is a partition $V(H)=U_{1} \cup \cdots \cup U_{4}$ with $\underset{i}{\min} |U_{i}| \geq \left(\frac{1}{4}-3\delta\right)k$ such that there are no red edges from $U_{1} \cup U_{2}$ to $U_{3} \cup U_{4}$ and no blue edges from $U_{1} \cup U_{3}$ to $U_{2} \cup U_{4}$. Then $G$ contains a monochromatic cycle of length at least $\left(1-59\delta\right)n$ and monochromatic cycles of length $\ell$ for all $\ell \in \left[4,\left\lceil\frac{1}{2}n\right\rceil\right]$.
\end{lemma}

We now prove Theorem \ref{Thm: main result} by applying the lemmas above to the reduced graph obtained from the Regularity Lemma.

\begin{proof}[Proof of Theorem \ref{Thm: main result}]
Choose $0<\delta <\frac{1}{144}$ and $d \ll \delta$ (where, as usual, $\ll$ means sufficiently smaller than). Let $\epsilon := \frac{1}{2}\epsilon\left(\frac{d}{2},2,2,\frac{d}{2}\right)$ be defined from $d$ as in the Blow-up Lemma; we may also assume that $\epsilon \ll d$ by taking a smaller $\epsilon$ if necessary. In particular, we may choose $\epsilon$ and $d$ so that we may apply the Embedding Lemma. Also, by Lemma \ref{Lemma: odd cycles}, we have, for any fixed integer $K$, that either $G \cong K_{n/4,n/4,n/4,n/4}$ and the colouring is a $2$-bipartite $2$-edge colouring or $G$ contains a monochromatic $C_{\ell}$ for all $\ell \in [4,K]$. Hence it is sufficient to prove that either $G \cong K_{n/4,n/4,n/4,n/4}$ and the colouring is a $2$-bipartite $2$-edge colouring, or there is some fixed integer $K$ such that $G$ contains a monochromatic $C_{\ell}$ for all $\ell \in \left[K, \left\lceil\frac{n}{2}\right\rceil \right]$.

Let $G$ be a graph of order $n$ with $\delta\left(G\right) \geq \frac{3}{4}n$. We apply the degree form of the $2$-colour Regularity Lemma to $G$, with parameters $d$ and $\epsilon$. Let $V_{0},V_{1},\ldots,V_{k}$ be the clusters (with $|V_{i}|=m$ for $i\geq 1$), and $G'$ be the subgraph of $G$ defined by Theorem \ref{Thm: regularity}. Let $H$ be the $(\epsilon,d)$-reduced graph defined from $G'$ earlier, with $2$-edge colouring $E(H)=E(R_{H})\cup E (B_{H})$.

We have $\delta\left(G'\right) \geq \left(\frac{3}{4}-d-\epsilon\right)n$. Suppose that $\delta\left(H\right) < \left(\frac{3}{4}-\delta\right)k$: then there is some $i\geq 1$ with $d_{H}\left(V_{i}\right)<\left(\frac{3}{4}-\delta\right)k$. For a vertex $v\in V_{i}$ , it has neighbours in $G'$ only in $V_{0}$, or in $V_{j}$ for those $j$ such that $v_{i}v_{j}$ is an edge of $H$. Hence 
\begin{align*}
d_{G'}(v) &< \left(\frac{3}{4}-\delta\right)km +|V_{0}|  \leq \left(\frac{3}{4}-\delta+\epsilon \right)n
\end{align*}
which is a contradiction, as $\delta \gg d + 2 \epsilon$. Hence $\delta\left(H\right) \geq \left(\frac{3}{4}-\delta\right)k$.

Applying Lemma \ref{Lemma: reduced graph}, we have one of the following.
\begin{enumerate}[(i)]
\item There is a component of $R_{H}$ or $B_{H}$ which contains a matching on at least $\left( \frac{2}{3} +\delta \right)k$ vertices.
\item There is a set $S$ of order at least $\left(\frac{2}{3}-\frac{\delta}{2}\right)k$ such that either $\Delta(R_{H}[S]) \leq 10\delta k$ or $\Delta(B_{H}[S]) \leq 10\delta k$.
\item There is a partition $V(H)=U_{1} \cup \cdots \cup U_{4}$ with $\underset{i}{\min} |U_{i}| \geq \left(\frac{1}{4}-3\delta\right)k$ such that there are no blue edges from $U_{1} \cup U_{2}$ to $U_{3} \cup U_{4}$ and no red edges from $U_{1} \cup U_{3}$ to $U_{2} \cup U_{4}$.
\end{enumerate}

If we are in the second or third case, we are done immediately by Lemma \ref{Lemma: large sparse set} and Lemma \ref{Lemma: pathological case} respectively. Hence we assume that there is a component of $R_{H}$ or $B_{H}$ which contains a matching on at least $\left( \frac{2}{3} +\delta \right)k$ vertices. By Lemma \ref{Lemma: large matching in large component or with odd cycle}, we may assume that there is a component $R'_{H}$ of $R_{H}$ which contains a matching on at least $\left( \frac{1}{2} +\delta \right)k$ vertices, and that either $R_{H}'$ contains an odd cycle or $|R_{H}'|\geq \left(1-5\delta\right)k$.

Take a matching in $R'_{H}$ with a maximal number of vertices. Let $r$ be the number of edges in the matching and $C_{1},C_{2},\ldots,C_{r}$ be the edges of the matching, with $C_{i}=v_{i,1}v_{i,2}$. For $2 \leq i \leq r$, let $P_{i}$ be a shortest path of $R_{H}$ from $C_{i}$ to $C_{1}$. We may asssume that the end-point of $P_{i}$ in $C_{i}$ is $v_{i,1}$. For all $i\geq 2$, let $v_{1,j_{i}}$ be the endpoint of $P_{i}$ in $C_{1}$. Note that the path $P_{i}$ may pass through vertices of $C_{j}$ for $j \neq i$.

We wish to apply the Blow-up Lemma to the clusters $V_{i,j}$ corresponding to vertices $v_{i,j}$. However, the Blow-up Lemma applies to super-regular pairs and currently we may only assume regularity. We show that by removing a small number of vertices, we may assume that all the edges of our odd-extended matching are super-regular pairs. Let $W_{1,1}$ be the set of vertices of $V_{1,1}$ with at least $\left(d-\epsilon\right)m$ neighbours in $V_{1,2}$ and $W_{1,2}$ be the set of vertices of $V_{1,2}$ with at least $\left(d-\epsilon\right)m$ neighbours in $V_{1,1}$. Then it is immediate from regularity that $|W_{1,j}| \geq (1-\epsilon)|V_{1,j}|$ for $j \in \{ 1,2\}$. We can check that $(W_{1,1},W_{1,2})$ is a $(\frac{3}{2}\epsilon,d-2\epsilon)$-super-regular pair. Similarly, we can obtain $W_{i,j}\subseteq V_{i,j}$ for all $i \geq 2$ and $j \in \{1,2\}$ such that $|W_{i,j}|\geq (1-\epsilon) |V_{i,j}|$ and $\left(W_{i,1},W_{i,2}\right)$ is $\left(\frac{3}{2}\epsilon,d-2\epsilon\right)$-super-regular for $R_{G}$.

Suppose first that there is an odd cycle in the component $R'_{H}$ of $R_{H}$. Either this cycle contains $v_{1,1}$, or there is a path from the cycle to $v_{1,1}$. Using Theorem \ref{Thm: embedding lemma}, we can find a red path $Q_{1}\subseteq R_{G}$  of odd length at most $2k$ between two vertices $v_{1}$ and $v_{2}$ of $W_{1,1}$. We can use the super-regular pair $(W_{1,1},W_{1,2})$ to find a red path $Q'_{1}$ of length four from $v_{1}$ to $v_{2}$, which does not intersect $Q_{1}$.

We now construct vertex-disjoint red paths $Q_{2}, Q_{2}',\ldots,Q_{r},Q_{r}'$ in $V\setminus\left(V(Q_{1})\cup V(Q_{2})\right)$ such that, for all $2\leq i \leq r$, both $Q_{i}$ and $Q'_{i}$ start in $W_{i,1}\subseteq V_{i,1}$ and pass through the clusters $V_{\ell}$ corresponding to vertices of $P_{i}$, before terminating in $W_{1,j_{i}}\subseteq V_{1,j_{1}}$. Indeed suppose that we have constructed such paths $Q_{2}, Q_{2}',\ldots,Q_{s-1},Q_{s-1}'$ for some $2\leq s \leq r$. For $i\geq1$, each path $Q_{i}$ or $Q_{i}'$ uses each cluster at most twice and so between them the paths $Q_{1},Q_{1}',Q_{2}, Q_{2}',\ldots,Q_{s-1},Q_{s-1}'$ contain at most $4(s-1)\leq 2k$ vertices in each cluster. Hence, we may remove the vertices of $Q_{1},Q_{1}',Q_{2}, Q_{2}',\ldots,Q_{s-1},Q_{s-1}'$ without affecting the regularity of pairs of clusters. Hence, we can find the paths $Q_{s}$ and $Q_{s}'$ by Theorem \ref{Thm: embedding lemma}.

Each path $Q_{1},Q_{1}',Q_{2}, Q_{2}',\ldots,Q_{r},Q_{r}'$ has at most two internal vertices contained in each $W_{i,j}$. Letting $W_{i,j}'$ be the vertices of $W_{i,j}$ not used as an internal vertex of some path $Q_{2}, Q_{2}',\ldots,Q_{r},Q_{r}'$ , we thus have $|W_{i,j}'|\geq |W_{i,j}|-2k \geq (1-3\epsilon)m$. Deleting vertices where appropriate, we may assume that each $W_{i,j}'$ has order $m'=\lceil(1-3\epsilon)m\rceil$. Note that the pairs $\left(W'_{i,1},W'_{i,2}\right)$ are $\left(2\epsilon,\frac{\delta}{2}\right)$-super-regular.

\begin{figure}[h!]
\centering
\includegraphics{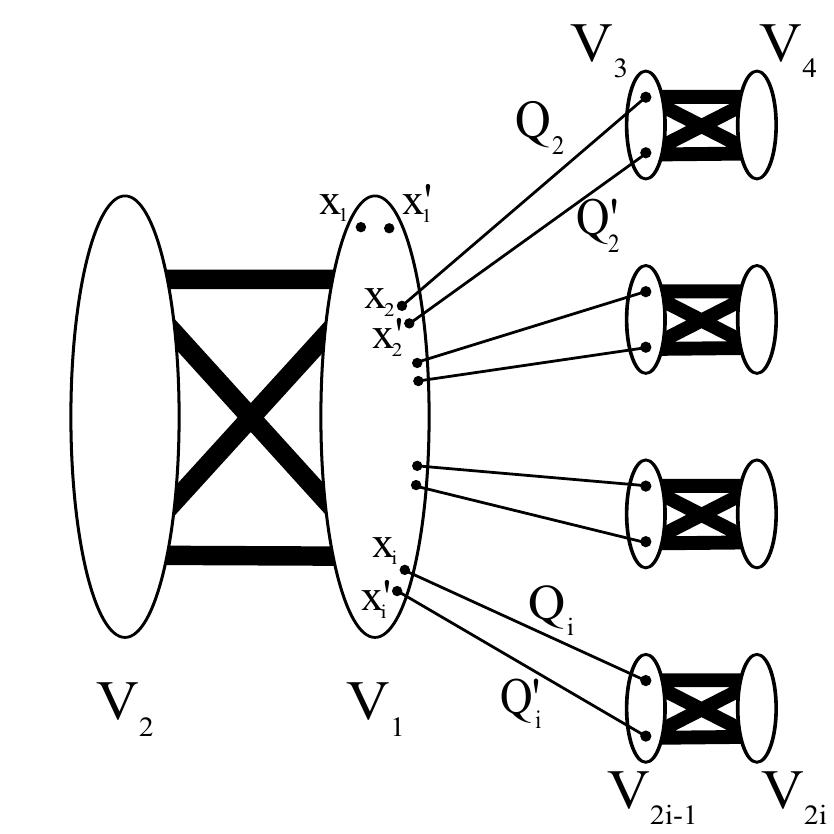}
\caption{The graph $G''$.}
\label{Fig: Applying blow-up}
\end{figure} 

Consider the subgraph $G''$ of $R_{G'}$ consisting of the super-regular pairs $\left(W'_{i,1},W'_{i,2}\right)$ and the paths $Q_{1},Q_{1'},Q_{2}, Q_{2}',\ldots,Q_{r},Q_{r}'$, as shown in Figure \ref{Fig: Applying blow-up}. (The super-regular pairs are shown with thick lines and the paths with thin lines.) Note that each of the paths contains at most $k$ vertices, and there are at most $k$ of them. However the union of the $W_{i,j}'$ has order at least $\left(\frac{1}{2}+\frac{\delta}{2}\right)n$. Let $K=3+|Q_{1}|$. If we replaced all of the super-regular pairs in $G''$ by complete bipartite graphs, it is clear that the resultant graph would contain $C_{\ell}$ for all $\ell \in \left[K, \left\lceil \frac{1}{2}n\right\rceil\right]$. By applying the Blow-up Lemma individually to each pair $\left(W'_{i,1},W'_{i,2}\right)$, we may thus embed $C_{\ell}$ into $ R_{G'}$ for all $\ell \in \left[K, \left\lceil \frac{1}{2}n\right\rceil\right]$. Note that we must sometimes restrict vertices of our embedding to lie in the neighbourhood of endvertices of some $Q_{\ell}$. However, there are only a bounded number ($O\left(k\right)$) of such restrictions, and they are all to sets of order at least $\frac{1}{2}dm'$, as we have made our pairs super-regular. Hence we are done in the case that $R_{H}'$ contains an odd cycle.

Suppose now that the component $R'_{H}$ of $R_{H}$ contains no odd cycles and hence $|R'_{H}|\geq \left(1-5\delta\right)k$. Then $R_{H}'$ is bipartite, with classes $H_{1}$ and $H_{2}$. Applying the Blow-up Lemma as above, we deduce that $C_{\ell} \subseteq R_{G'}$ for all even $\ell \in \left[4,  \left(\frac{1}{2}+\frac{\delta}{2}\right)n \right]$. Hence we are done if we can show that there is a fixed integer $K$ such that $G$ contains a monochromatic $C_{\ell}$ for all odd $\ell \in \left[K,\left\lceil \frac{1}{2}n\right\rceil \right]$.

As $R'_{H}$ is a connected component with order at most $k$ there is a red path of length at most $k-1$ between any pair of vertices in $R_{H}'$. Let $X$ be the union of all the clusters in $H_{1}$, and $Y$ be the union of all the clusters in $H_{2}$, so that $X$ and $Y$ are subsets of $V$. Using the Blow-up Lemma as above we see that, after removing at most $\epsilon|X|$ vertices from $X$ and $\epsilon|Y|$ vertices from $Y$, the following holds.
\begin{itemize}
\item Between any vertices $u,v\in X$ there are, in $R_{G'}[X\cup Y]$, paths of length $\ell$ for all even $\ell \in [2k,\left(\frac{1}{2} + \frac{ \delta}{2}\right)n]$ (and similarly for $Y$).
\item Between any $u\in X$ and any $v \in Y$ there are, in $R_{G'}[X \cup Y]$, paths of length $\ell$ for all odd $\ell \in [2k-1,\left(\frac{1}{2} + \frac{ \delta}{2}\right)n]$.
\end{itemize}

If either $X$ or $Y$ contains an internal red edge, then we have cycles of all odd lengths between $2k+1$ and $\left(\frac{1}{2} + \delta\right)n$. Hence we may assume that $R_{G}[X \cup Y]$ is bipartite. Then
\begin{align*}
|X \cup Y | &\geq \left(1-5\delta\right)km(1-\epsilon) \\
&\geq (1-6\delta )n.
\end{align*}
However, if $\max \{|X|,|Y|\} >\frac{1}{2}n$, we are done by Lemma \ref{Lemma: no large independent sets}. Hence we may assume that $\min \{ |X| , |Y| \} \geq \left(\frac{1}{2}-6\delta\right)n$.

If any vertex $v$ of $V \setminus (X \cup Y)$ has at least one red neighbour in both $X$ and $Y$, then using the paths between a red neighbour of $v$ in $X$ and a red neighbour of $v$ in $Y$, we have cycles of all odd lengths between $2k+1$ and $\left(\frac{1}{2} + \delta\right)n $. Hence all vertices of $V \setminus (X \cup Y)$ have no red neighbours in at least one of $X$ or $Y$.

Define disjoint sets $X'$ and $Y'$ by letting $X'$ be the set of vertices of $V \setminus (X \cup Y)$ with at least two red neighbours in $Y$, and $Y'$ be the set of vertices of $V \setminus (X \cup Y)$ with at least two red neighbours in $X$. Then there are no red edges between $X'$ and $X$ or between $Y'$ and $Y$. If there is a red edge $uv$ within $X'$, let $u'$ and $v'$ be distinct vertices of $Y$ with $uu'$ and $vv'$ both red edges. Then $u'uvv'$ is a red path of length three between vertices of $Y$, with internal vertices in $V\setminus \left(X\cup Y\right)$. Using the $u'$-$v'$ paths obtained from the Blow-up Lemma, we have red cycles of all odd lengths between $2k+3$ and $\left(\frac{1}{2} + \delta\right)n$.

So we may assume that $R_{G}[X \cup X' \cup Y \cup Y']$ is bipartite with classes $X \cup X'$ and $Y \cup Y'$. We may assume that both $X\cup X'$ and $Y \cup Y'$ have order at most $\frac{1}{2}n$, else we are done by Lemma \ref{Lemma: no large independent sets}.

A vertex not in $X \cup X' \cup Y \cup Y'$ has at least $\left(\frac{3}{4}-6 \delta \right)n$ neighbours in $X \cup Y$. Let $X''$ be the set of vertices not in $X \cup X' \cup Y \cup Y'$ with at least $\left(\frac{3}{8}-3\delta\right)n$ neighbours in $X$, and $Y''$ be the set of vertices not in $X \cup X' \cup Y \cup Y'$ with at least $\left(\frac{3}{8}-3\delta\right)n$ neighbours in $Y$. Letting $X_{0}=X \cup X' \cup X''$ and $Y_{0}=Y \cup Y' \cup Y''$ we see that $V$ is the (not necessarily disjoint) union of $X_{0}$ and $Y_{0}$.

Without loss of generality, we may suppose that $|X_{0}| \geq \frac{1}{2}n$. By definition, all vertices in $X''$ have at least $\left(\frac{3}{8}-3\delta\right)n$ neighbours in $X$. However vertices in $X''$ have at most one red neighbour in $X$, else they would have been in $Y'$. All vertices in $X \cup X'$ have at most $\frac{1}{4}n$ non-neighbours in $G$ and so at least $|X_{0}|- \frac{1}{4}n $ neighbours in $X_{0}$. As there at most $|X''|$ red edges between $X''$ and $X$, the set $X'''$ of vertices in $X$ with a red neighbour in $X''$ has order at most $|X''|$. Vertices in $X \setminus X'''$ have no red neighbours in $X_{0}$, while vertices in $X' \cup X'''$ have no red neighbours in $X_{0} \setminus X''$. Hence
\begin{align}
\label{eqn: degrees in X_{0}}
d_{B[X_{0}]} (v) \geq \begin{cases}
|X_{0}|-\frac{1}{4}n &\quad v \in X \setminus X''' \\
|X_{0}|-\frac{1}{4}n- |X''| &\quad v \in X''' \cup X' \\
\left(\frac{3}{8}-3\delta\right)n -1 &\quad v \in X'' .
\end{cases}
\end{align}
Since $|X' \cup X'' \cup X'''| \leq 6\delta n$, the conditions of Theorem \ref{Thm: chvatal} are satisfied on the graph $B_{G}[X_{0}]$ and so $B_{G}[X_{0}]$ is Hamiltonian.

However, using \eqref{eqn: degrees in X_{0}}, we have
\begin{align*}
e\left(B_{G}[X_{0}]\right) &\geq  \frac{1}{2}\left(|X_{0}|-\frac{1}{4}n\right)\left|X\setminus X'''\right|+ \left(\left(\frac{3}{16}-\frac{3\delta}{2}\right)n-\frac{1}{2} \right)|X''| \\
&\qquad + \frac{1}{2}\left(|X_{0}|-\frac{1}{4}n-|X''|\right)\left(|X'''| +| X'|\right)\\
&= \frac{1}{2}|X_{0}|\left(|X_{0}|-\frac{1}{4}n\right)  \\
&\qquad +\left(\left(\frac{3}{16}-\frac{3\delta}{2}\right)n-\frac{1}{2}-\frac{1}{4}|X_{0}|-\frac{1}{2}\left(|X''' |+| X'|\right)\right)|X''| \\
&\geq \frac{1}{4}|X_{0}|^{2} +  \left(\left(\frac{1}{16}-6\delta\right)n-\frac{1}{2}\right)|X''|.
\end{align*}
Here we have used
\begin{align*}
\frac{1}{4}|X_{0}|+\frac{1}{2}\left(|X''' |+| X'|\right) &\leq \frac{1}{4}\left(|X|+|X'|\right) +\frac{1}{2}|X'| + \frac{3}{4}|X''| \\
&\leq \frac{1}{8}n + \frac{3}{4} |V\setminus \left(X \cup Y\right) |\\
&\leq \left(\frac{1}{8}+\frac{9\delta}{2}\right)n.
\end{align*}

Hence, from Theorem \ref{Thm: bondy} we see that either $B_{G}[X_{0}]$ is pancyclic, in which case $C_{\ell}\subseteq B_{G}$ for all $\ell \in [3,|X_{0}|]$, or $B_{G}[X_{0}]\cong K_{|X_{0}|/2,|X_{0}|/2}$ and $e(B_{G}[X_{0}])=\frac{1}{4}|X_{0}|^{2}$. Hence, in the latter case, $X''=\emptyset$. Similarly, if $|Y_{0}| \geq \frac{1}{2}n$, then either $B_{G}[Y_{0}]$ is pancyclic, or $Y''=\emptyset$. Hence we may assume that $X''=Y''=\emptyset$ and hence $B_{G}$ is bipartite. Thus by Lemma \ref{Lemma: no large independent sets}, we are done.

If we are in the second or third case of Lemma \ref{Lemma: reduced graph}, we are done by Lemma \ref{Lemma: large sparse set} and Lemma \ref{Lemma: pathological case} respectively.
\end{proof}

\section{Proof of Lemmas}
\label{Sec: proof of extremal lemma}

In this section we shall prove the lemmas used in the proof of Theorem \ref{Thm: main result}. Later in the section, we will prove Lemma \ref{Lemma: large sparse set} and Lemma \ref{Lemma: pathological case}. However, we begin with the proofs of Lemma \ref{Lemma: reduced graph} and Lemma \ref{Lemma: large matching in large component or with odd cycle}. Throughout both proofs we shall assume that $R'$ is a largest component of $R$, and that $B'$ is a largest component of $B$. We let $W_{1} =V(B') \cap V(R')$, $W_{2}=V(R')\setminus V(B')$, $W_{3} = V(B') \setminus V(R')$ and $W_{4} = V -(W_{1} \cup W_{2} \cup W_{3})$. We will need the following claim about the component structure.

\begin{claim}
\label{Claim: components cover G}
Let $0<\delta < \frac{1}{36} $ and let $G$ be a graph of sufficiently large order $n$ with $\delta(G) \geq \left( \frac{3}{4} -\delta \right)n$. Suppose that we are given a $2$-edge colouring $E(G)=E(R) \cup E(B)$. Then one of the following holds.
\begin{itemize}
\item One of $R$ or $B$ is connected.
\item $V(G)=V(R') \cup V(B')$ and both $R'$ and $B'$ have order at least $\left(\frac{3}{4}-\delta \right)n$.
\item There is a partition $V(G)=U_{1} \cup \cdots \cup U_{4}$ with $\underset{i}{\min} |U_{i}| \geq \left(\frac{1}{4}-3\delta\right)n$ such that there are no red edges from $U_{1} \cup U_{2}$ to $U_{3} \cup U_{4}$ and no blue edges from $U_{1} \cup U_{3}$ to $U_{2} \cup U_{4}$.
\end{itemize}
\end{claim}
\begin{proof}
If neither of the above statements holds, then both $R$ and $B$ are disconnected. Suppose first that $|R'|\leq \left(\frac{5}{12} - \delta\right)n$. Then $\Delta(R) < \left(\frac{5}{12} - \delta\right)n$ and hence $\delta(B) > \frac{1}{3}n$. Since $B$ is disconnected, we see that $B$ has exactly two components $B_{1}$ and $B_{2}$ with $\frac{1}{3} n< |B_{2}| \leq |B_{1}| < \frac{2}{3}n$. Then $W_{i} =V(B_{i}) \cap V(R')$, for $i \in \{1,2\}$.

Suppose that $W_{i} \neq \emptyset$, for some $i \in \{1,2\}$. Let $v \in W_{i}$. Then $v$ has no neighbours outside $R' \cup B_{i}$, and so $\Gamma_{G} \left(v \right) \subseteq R' \cup B_{i} $. Hence, as $W_{3-i}= R' \setminus B_{i}$, 
\begin{align*}
|W_{3-i}| &\geq \left|\Gamma_{G} \left(v \right) \right| - |B_{i}| > \left(\frac{1}{12}-\delta\right)n.
\end{align*}
In particular $W_{3-i} \neq \emptyset$, and so $W_{1}$ is non-empty if and only if $W_{2}$ is non-empty. As $V(R')=W_{1} \cup W_{2}$, we see that both $W_{1}$ and $W_{2}$ are therefore non-empty.

But then
\begin{align*}
|R'| =|W_{1}|+|W_{2}| &\geq \left(\frac{3}{4}-\delta\right)n -|B_{1}| + \left(\frac{3}{4}-\delta\right)n -|B_{2}| \\
&= \left(\frac{1}{2}-2\delta \right)n.
\end{align*}
This contradicts our assumption. We may therefore assume that $|R'|> \left(\frac{5}{12} - \delta\right)n$, and similarly $|B'|> \left(\frac{5}{12} - \delta\right)n$.

Note that there are no edges (of either colour) between $W_{1}$ and $W_{4}$ or between $W_{2}$ and $W_{3}$. If $W_{4} = \emptyset$, then $V(G)=V(R') \cup V(B')$. As neither $R$ nor $B$ is connected, we must have that $W_{2}$ is non-empty. Let $v \in W_{2}$: since $\Gamma_{G}(v) \cap W_{3} = \emptyset$, we see that $|W_{3}| \leq \left(\frac{1}{4} + \delta \right)n$ and so $|R'| \geq \left(\frac{3}{4}-\delta \right)n$. We may similarly show that $|B'| \geq \left(\frac{3}{4}-\delta \right)n$.

If, however, $W_{4}\neq \emptyset$, choose $x \in W_{4}$. As $\Gamma_{G}(x) \cap W_{1}=\emptyset$, we have $|W_{1}| \leq \left(\frac{1}{4}+\delta\right)n$. However both $R'$ and $B'$ have order at least $\left(\frac{5}{12}-\delta\right)n$ and hence both $W_{2}$ and $W_{3}$ are non-empty. Thus, arguing as for $W_{4}$, we see that both $W_{2}$ and $W_{3}$ have order at most $\left(\frac{1}{4}+\delta\right)n$ and so $W_{1}$ is non-empty. This in turn implies that $W_{4}$ has order at most $\left(\frac{1}{4}+\delta\right)n$. Hence each $W_{i}$ has order at least $\left(\frac{1}{4}-3\delta\right)n$.
\end{proof}

\begin{proof}[Proof of Lemma \ref{Lemma: reduced graph}]
We assume throughout that $n$ is sufficiently large. Let $0 < \delta < \frac{1}{36}$. Suppose that $G$ is a graph of order $n$ with $\delta(G) \geq \left( \frac{3}{4} -\delta \right)n$ and that we are given a $2$-edge colouring $E(G)=E(R) \cup E(B)$.

If $\underset{i}{\min} |W_{i}| \geq \left(\frac{1}{4} -3 \delta \right)n$, then we are in the third case of Lemma \ref{Lemma: reduced graph}. Hence, we assume throughout that 
\begin{align}
\label{Eqn: upper bound on Wi}
\underset{i}{\min} |W_{i}| < \left(\frac{1}{4} -3 \delta \right)n.
\end{align}
We will also assume that neither $R'$ nor $B'$ contains a matching on at least $\left(\frac{2}{3}+\delta\right)n$ vertices, and refer to this as our main assumption.

We are aiming to show that there is a large set on which one of the colours has a very low density. We show that the orders of $B'$ and $R'$ can not take certain values.

\begin{claim}\label{Claim: B' of 'medium' size}
Either $|B'| < \left(\frac{1}{3}+\frac{\delta}{2}\right)n$ or $|B'| > \left(\frac{2}{3}-\frac{\delta}{2}\right)n$.
\end{claim}
\begin{proof}
Suppose that $\left(\frac{1}{3}+\frac{\delta}{2}\right)n \leq |B'|\leq \left(\frac{2}{3}-\frac{\delta}{2}\right)n$. Then $R$ is connected, by Claim \ref{Claim: components cover G}.

Let $V_{1}$ be the smaller of $V(B')$ and $V\setminus V(B')$. Let $V_{2}=V \setminus V_{1}$ and $F$ be the bipartite graph between $V_{1}$ and $V_{2}$. There are no blue edges between $V_{1}$ and $V_{2}$ and so all edges of $F$ are red. For a subset $S$ of $V_{1}$ we shall find a lower bound on $|\Gamma_{F}(S)|$ by splitting into the cases that $|S|> \left(\frac{1}{4}+\delta\right)n$ and $|S|\leq \left(\frac{1}{4} + \delta \right)n$.

If $S \subseteq V_{1}$ and $|S|>\left(\frac{1}{4}+\delta\right)n$, consider a vertex $v \in V_{2}$. Then, as $d_{G}\left(v\right) \geq \left(\frac{3}{4}-\delta\right)n$, $v$ must have a neighbour in $S$. Hence $\Gamma_{F}(S)=V_{2}$, and so $|\Gamma_{F}(S)| =|V_{2}| \geq |V_{1}| \geq |S|$.

If $S \subseteq V_{1}$ and $|S|\leq \left(\frac{1}{4}+\delta\right)n$, then any vertex in $S$ has at least $|V_{2}|-\left(\frac{1}{4}+\delta\right)n$ neighbours in $V_{2}$. Hence
\begin{align*}
|\Gamma_{F}(S)| &\geq |V_{2}|- \left(\frac{1}{4}+\delta\right)n \\
&=|S| -\left( \left(\frac{1}{4}+\delta\right)n+|S| -|V_{2}|\right) \\
&\geq |S|- \left( \left(\frac{1}{2}+2\delta\right)n -|V_{2}|\right).
\end{align*}

Thus by the defect form of Hall's Theorem, $F$ contains a matching with at least
\begin{displaymath}
|V_{1}| - \max \left\{0,\left( \left(\frac{1}{2}+2\delta\right)n -|V_{2}|\right)\right\}
\end{displaymath}
edges. As $|V_{1}|+|V_{2}|=n$ and $|V_{1}| \geq \left(\frac{1}{3}+\frac{\delta}{2}\right)n$, this matching contains at least $\left( \frac{2}{3} +\delta \right)n$ vertices. As $R$ is connected, this contradicts our main assumption.
\end{proof}

By Claim \ref{Claim: components cover G} and the assumption \eqref{Eqn: upper bound on Wi}, we may assume that either one of $R$ or $B$ is connected or $V(G)=V(R') \cup V(B')$ and $\min \{ |V(R')|,|V(B')|\} \geq \left(\frac{3}{4}-\delta\right)n$. In either case, there will be a monochromatic component of order at least $\left(\frac{3}{4}-\delta \right)n$. Let $X_{R}=V\setminus V(R')$ and $X_{B}=V\setminus V(B')$. We make the following definitions when $R'$ or $B'$ is large.
\begin{defn}
Fix an integer \emph{$t\geq \delta^{-1}$}. Suppose that $|R'| \geq \left(\frac{2}{3}+\delta\right)n$.
\begin{itemize}
\item Let \emph{$S_{R} \subseteq V(R')$} be a set such that
\begin{displaymath}
q\left(R\left[V(R')-S_{R}\right]\right) > |S_{R}| + |V(R')|- \left(\frac{2}{3} + \delta \right)n.
\end{displaymath}
Note that, in view of our main assumption, such a set exists by Theorem \ref{Thm: Tutte}.
\item For $1\leq i \leq n$, let $T_{R,i}$ be the set of vertices which lie in components of $R\left[V(R')-S_{R}\right]$ of order $i$.
\item Let $T_{R}=\bigcup_{1\leq i \leq t}T_{R,i}$.
\end{itemize}
If $|B'| \geq \left(\frac{2}{3}+\delta\right)n$, we define $S_{B}$, $T_{B,i}$ and $T_{B}$ similarly.
\end{defn}

We shall use the following result throughout. Note that, as with Claim \ref{Claim: B' of 'medium' size}, we may exchange the roles of $R$ and $B$ to obtain a symmetrical version of this result.
\begin{claim}
\label{Claim: greatly separating set}
Suppose that $|V(R')| \geq \left(\frac{2}{3}+\delta\right)n$. Then $|S_{R}| < \left(\frac{1}{3} + \frac{1}{2}\delta\right)n$ and 
\begin{displaymath}
|X_{R} \cup T_{R}| >|S_{R}| +\left(\frac{1}{3}-2\delta\right)n.
\end{displaymath}
Further, if $C_{B}$ is a component of $B$ with $|C_{B}|\leq \left(\frac{5}{12} - 2 \delta \right)n$, then $C_{B} \subseteq S_{R}$.
\end{claim}
\begin{proof}
All vertices of $V(R')$ lie in $S_{R}$ or some component of $R\left[V(R')-S_{R}\right]$. Hence
\begin{align*}
 |V(R')| &\geq |S_{R}|+ q\left(R\left[V(R')-S_{R}\right]\right) \\
&> 2 |S_{R}| +|V(R')|- \left(\frac{2}{3} + \delta \right)n.
\end{align*}
This implies that $|S_{R}| < \left(\frac{1}{3} + \frac{1}{2}\delta\right)n$.

There are at most $|T_{R}|$ components of $R\left[V(R')-S_{R}\right]$ of order at most $t$. However, there are at most $\frac{1}{t}n\leq \delta n$ components of $R\left[V(R')-S_{R}\right]$ of order at least $t$. Hence $|T_{R}| \geq q\left(R\left[V(R')-S_{R}\right]\right) - \delta n$. As $X_{R}$ and $T_{R}$ are disjoint, we have
\begin{align*}
|X_{R}\cup T_{R}| &\geq n-|V(R')|+ q\left(R\left[V(R')-S_{R}\right]\right) -\delta n \\
&> \left(1-\delta\right)n -|V(R')|+|S_{R}| +|V(R')|- \left(\frac{2}{3} + \delta \right)n \\
&= |S_{R}| +\left(\frac{1}{3}-2\delta\right)n.
\end{align*}

Finally, suppose that $C_{B}$ is a component of $B$ with $|C_{B}| \leq \left(\frac{5}{12} - 2 \delta \right)n$. A vertex in $C_{B}$ has blue degree at most $|C_{B}|-1$. Hence any vertex in $C_{B}$ must have red degree at least
\begin{align}
\nonumber \delta(G) - |C_{B}|+1  &\geq \left(\frac{3}{4}-\delta\right)n - \left(\frac{5}{12} - 2 \delta \right)n+1 \\
\label{Eqn: red degree of vtx in small blue cmpt} &= \left(\frac{1}{3} +\delta \right) n +1.
\end{align}

A vertex in $X_{R}$ has red degree at most
\begin{displaymath}
|X_{R}|-1 \leq \left(\frac{1}{3}-\delta \right)n-1.
\end{displaymath}
However, a vertex in $T_{R}$ is in a component of $R\left[V(R')-S_{R}\right]$ of order at most $t$. Hence, for all $v \in T_{R}$,
\begin{align*}
d_{R}(v) &\leq t+|S_{R}| \\
&< \left(\frac{1}{3} + \frac{1}{2}\delta\right)n+t.
\end{align*}
Hence \eqref{Eqn: red degree of vtx in small blue cmpt} and $\delta(G) \geq \left(\frac{3}{4}-\delta \right)n$ imply that $ C_{B}\cap\left(X_{R} \cup T_{R}\right) =\emptyset$.

Suppose that there exists $v \in C_{B} \setminus S_{R}$. A blue neighbour of $v$ lies in $C_{B}$, and so $v$ has no blue neighbours in $ X_{R} \cup T_{R}$. However, $C_{B} \subseteq V(R')$ and so $v$ has no red neighbours in $X_{R}$. The only vertices with red neighbours in $T_{R}$ are those in $S_{R} \cup T_{R}$, and so we see that $v$ also has no red neighbours in $T_{R}$. Hence $v$ has no neighbours in $X_{R} \cup T_{R}$, and so
\begin{align*}
d_{G}(v) &\leq n - |X_{R} \cup T_{R}|\\
&< \left(\frac{2}{3}+2\delta\right)n.
\end{align*}
This contradicts $\delta(G) \geq \left(\frac{3}{4}-\delta \right)n$, and so $C_{B} \subseteq S_{R}$.
\end{proof}

We may thus assume that $S_{R}$ is not much bigger than $\frac{1}{3}n$. The following result shows that if $S_{R}$ has order approaching $\frac{1}{3}n$ and $R'$ is very large, $R[V\setminus S_{R}]$ is the required graph with low density.
\begin{claim}
\label{Claim: S is large gives sparse set}
Suppose that $|V(R')| \geq \left(1-\frac{5\delta}{2}\right)n$ and that $|S_{R}|\geq \left(\frac{1}{3}-2\delta\right)n$. Then $R[V\setminus S_{R}]$ is a graph on at least $\left(\frac{2}{3}-\frac{\delta}{2}\right)n$ vertices with maximum degree at most $10\delta n$.
\end{claim}
\begin{proof}
That $R[V\setminus S_{R}]$ is a graph of order at least $\left(\frac{2}{3}-\frac{\delta}{2}\right)n$ follows immediately from Claim \ref{Claim: greatly separating set}.

For all $1\leq i \leq n$, there are exactly $\frac{1}{i}|T_{R,i}|$ components of order $i$ in $R\left[V(R')-S_{R}\right]$. Hence
\begin{align*}
\sum_{i\geq 1} \frac{1}{2i-1}|T_{R,2i-1}| &= q\left(R\left[V(R')-S_{R}\right]\right) \\
&> |S_{R}| + |V(R')|- \left(\frac{2}{3} + \delta \right)n.
\end{align*}
However,
\begin{align*}
\sum_{i\geq 1} \frac{1}{2i-1}|T_{R,2i-1}| &\leq |T_{R,1}| + \frac{1}{3} \sum_{i\geq 2} |T_{R,2i-1}| \\
&\leq |T_{R,1}| +\frac{1}{3}\left( |V(R')|- |S_{R}| -|T_{R,1}|\right).
\end{align*}
Combining these inequalities, and using the bounds on $|V(R')|$ and $|S_{R}|$, we have
\begin{align*}
\frac{2}{3}|T_{R,1}| &> \frac{4}{3} |S_{R}| + \frac{2}{3}|V(R')| - \left(\frac{2}{3} + \delta \right)n \\
&\geq \frac{4}{3} \left(\frac{1}{3}-2\delta\right)n + \frac{2}{3}\left(1-\frac{5\delta}{2}\right)n - \left(\frac{2}{3} + \delta \right)n \\
&= \left(\frac{4}{9} -\frac{16\delta}{3} \right)n.
\end{align*}
Hence $|T_{R,1}| > \left(\frac{2}{3} -8\delta\right) n$.

However, $T_{R,1}$ is a set of isolated vertices in $R[V\setminus S_{R}]$. As $|V\setminus S_{R}| \leq \left(\frac{2}{3}+2\delta\right)n$, we see that $R[V\setminus S_{R}]$ has maximum degree at most $10\delta n$.
\end{proof}

We may now complete the proof of Lemma \ref{Lemma: reduced graph}, using the preceding claims. By Claim \ref{Claim: components cover G}, we may without loss of generality assume that $|V(R')|\geq \left( \frac{3}{4} -\delta \right)n$. We divide into cases depending on the order of $B'$.

If $|B'| < \left(\frac{1}{3} + \frac{\delta}{2} \right)n$, then any component of $B$ has order at most $\left(\frac{1}{3} + \frac{\delta}{2} \right)n$. By Claim \ref{Claim: greatly separating set}, $|S_{R}| < \left(\frac{1}{3} + \frac{1}{2}\delta\right)n$ and $S_{R}$ contains any small blue components. Thus $S_{R}$ contains all components of $B$, and hence has order $n$, a contradiction.

We can not have $\left(\frac{1}{3}+\frac{\delta}{2}\right)n \leq |B'|\leq \left(\frac{2}{3}-\frac{\delta}{2}\right)n$ by Claim \ref{Claim: B' of 'medium' size}.

If $\left(\frac{2}{3}-\frac{\delta}{2}\right)n< |B'| < \left(\frac{2}{3} + \delta\right)n$, then, by Claim \ref{Claim: components cover G}, $R$ is connected. Also, all components of $B$ other than $B'$ have order at most $\left(\frac{1}{3}+\frac{\delta}{2}\right)n$. Hence, by Claim \ref{Claim: greatly separating set}, $S_{R}$ contains $X_{B}$ and so $|S_{R}|> \left(\frac{1}{3}-\delta \right)n$. Thus we are done by Claim \ref{Claim: S is large gives sparse set}.

Finally, suppose that $|B'|\geq \left(\frac{2}{3}+ \delta \right)n$. We may assume that there is a set $S_{B} \subseteq V(B')$ such that
\begin{displaymath}
q\left(B\left[V(B')-S_{B}\right]\right) > |S_{B}| + |V(B')|- \left(\frac{2}{3} + \delta \right)n.
\end{displaymath}
By Claim \ref{Claim: greatly separating set}, we see that $X_{R} \subseteq S_{B}$ and $X_{B} \subseteq S_{R}$.

Suppose that there is a vertex $v \in T_{R} \cap T_{B}$. Then $v$ has at most $|S_{R}|+t$ red neighbours and at most $|S_{B}|+t$ blue neighbours. As $|S_{R}|$ and $|S_{B}|$ both have order at most $\left(\frac{1}{3} +\frac{\delta}{2}\right)n$, this contradicts the minimal degree of $G$. Hence $T_{R} \cap T_{B}=\emptyset$.

Suppose that $T_{B}\setminus S_{R}$ is non-empty and let $v\in T_{B}\setminus S_{R}$. As $X_{R} \subseteq S_{B}$, we have $T_{B}\subseteq  V(R')$. Hence $v$ has no red neighbours in $X_{R}$. Vertices in $T_{R}$ only have red neighbours in $T_{R} \cup S_{R}$. However, $T_{B} \cap T_{R} =\emptyset$ and so $v \notin T_{R} \cup S_{R}$. In particular $v$ has no red neighbours in $X_{R} \cup T_{R}$.

Hence, $v$ has at least $|X_{R}\cup T_{R}|-\left(\frac{1}{4}+\delta \right)n$ blue neighbours in $X_{R} \cup T_{R}$, as $v$ has at most $\left(\frac{1}{4}+\delta\right) n$ non-neighbours. However $v\in T_{B}$ and so all but $t$ of its blue neighbours are in $S_{B}$. Hence
\begin{align*}
|S_{B}| \geq |X_{R} \cup T_{R}|-\left(\frac{1}{4}+\delta \right)n-t > |S_{R}| +\left(\frac{1}{12}-3\delta\right)n-t,
\end{align*}
where the second inequality uses Claim \ref{Claim: greatly separating set}.

Similarly, if $T_{R} \setminus S_{B}$ is non-empty, then
\begin{displaymath}
|S_{R}| > |S_{B}| +\left(\frac{1}{12}-3\delta\right)n-t.
\end{displaymath}
As these can not both occur, one of $T_{R}\setminus S_{B}$ or $T_{B} \setminus S_{R}$ is empty.

We assume without loss of generality that $T_{B} \subseteq S_{R}$. Then $S_{R}$ contains the disjoint sets $T_{B}$ and $X_{B}$. Hence, using Claim \ref{Claim: greatly separating set}, again
\begin{displaymath}
|S_{R}| \geq |T_{B}\cup X_{B}| > |S_{B}|+\left(\frac{1}{3}-2\delta\right)n.
\end{displaymath}
Thus $|S_{R}| \geq \left(\frac{1}{3}-2\delta\right)n$. As $|S_{R}|<\left(\frac{1}{3} + \frac{1}{2}\delta\right)n$, we must have $|S_{B}| \leq \frac{5\delta}{2}n$. As $X_{R}\subseteq S_{B}$, we see that $|V(R')| \geq \left(1-\frac{5\delta}{2}\right)n$. Hence, by Claim \ref{Claim: S is large gives sparse set}, we are done.
\end{proof}

We now prove Lemma \ref{Lemma: large matching in large component or with odd cycle}, using similar methods to those used in the proof of Lemma \ref{Lemma: reduced graph}.

\begin{proof}[Proof of Lemma \ref{Lemma: large matching in large component or with odd cycle}]
We assume throughout that $n$ is sufficiently large. Let $0 < \delta < \frac{1}{6}$. Suppose that $G$ is a graph of order $n$ with $\delta(G) \geq \left( \frac{3}{4} -\delta \right)n$ and that we are given a $2$-edge colouring $E(G)=E(R) \cup E(B)$.

Suppose that $R'$ contains a matching on at least $\left(\frac{2}{3}+\delta\right)n$ vertices. We may assume that $|V(R')|<\left(1-5\delta\right)n$ and $R'$ is bipartite with classes $Y_{R}$ and $Z_{R}$, otherwise we are done. Without loss of generality, we assume that $|Z_{R}| \geq |Y_{R}|$ and so $|Y_{R}| \leq \frac{1}{2}|V(R')| < \left(\frac{1}{2}-\frac{5\delta}{2}\right)n$. As each edge of the matching contains one vertex from $Z_{R}$ and one from $Y_{R}$, we have
\begin{align}
\label{Eqn: Order of Y and Z} \left(\frac{1}{3}+\frac{\delta}{2}\right)n \leq |Y_{R}| \leq |Z_{R}| < \left(\frac{2}{3}-\frac{11\delta}{2}\right)n.
\end{align}

As in the proof of Lemma \ref{Lemma: reduced graph}, we let $X_{R} = V \setminus V(R')$ and $X_{B} = V \setminus V(B')$. Note that $5\delta n< |X_{R}| \leq \left(\frac{1}{3}-\delta\right)n$. As $|V(R')| \geq \left(\frac{2}{3} + \delta\right)n$, we are not in case (iii) of Claim \ref{Claim: components cover G}. Hence we may assume that $|X_{B}| \leq \left(\frac{1}{4}-\delta \right)n$ and $X_{B} \cap X_{R} = \emptyset$.

We will first show that $B'$ contains a matching on at least $\left( \frac{1}{2} +\delta \right)n$ vertices. Suppose not; then by Theorem \ref{Thm: Tutte} there is a set $S \subseteq V(B')$ such that
\begin{displaymath}
q\left(B\left[V(B')-S\right]\right) > |S| + |V(B')| - \left(\frac{1}{2} + \delta \right)n.
\end{displaymath}
We will apply the same arguments as used in Claim \ref{Claim: greatly separating set} to the set $S$.

All vertices of $V(B')$ lie in $S$ or some component of $B\left[V(B')-S\right]$. Hence
\begin{align*}
 |V(B')| &\geq |S|+ q\left(B\left[V(B')-S\right]\right) \\
&> 2 |S| +|V(B')|- \left(\frac{1}{2} + \delta \right)n.
\end{align*}
This implies that $|S| < \left(\frac{1}{4} + \frac{1}{2}\delta\right)n$.

Let $t$ be an integer with $t \geq \delta^{-1}$. We let $T$ be the set of vertices in components of $B\left[V(B')-S\right]$ with order at most $t$. Then $|T|\geq q\left(B\left[V(B')-S\right]\right)-\delta n$.

Any vertex in $T$ has blue degree at most 
\begin{displaymath}
|S|+t \leq \left(\frac{1}{4}+\frac{\delta}{2}\right)n +t
\end{displaymath}
and any vertex in $X_{B}$ has blue degree at most 
\begin{displaymath}
|X_{B}|-1 \leq \left(\frac{1}{4}+\delta\right)n-1.
\end{displaymath}
Also any vertex in $X_{R}$ has red degree at most
\begin{displaymath}
|X_{R}|-1 \leq \left(\frac{1}{3}-\delta\right)n-1
\end{displaymath}
and any vertex in $Z_{R}$ has red degree at most
\begin{displaymath}
|Y_{R}| < \left(\frac{1}{2}-\frac{5\delta}{2}\right)n.
\end{displaymath}
Hence any vertex in the intersection of $T \cup X_{B}$ and $Z_{R} \cup X_{R}$ has degree at most $\left(\frac{3}{4}-\frac{3\delta}{2}\right)n-1$. As $\delta(G) \geq \left(\frac{3}{4}-\delta \right)n$, we deduce that $T \cup X_{B}$ does not intersect $Z_{R} \cup X_{R}$. 

Hence $T \cup X_{B} \subseteq Y_{R}$. However, $T$ and $X_{B}$ are disjoint sets, and so
\begin{align*}
|Y_{R}| &\geq |T\cup X_{B}| \\
&\geq q\left(B\left[V(B')-S\right]\right)-\delta n+n -|V(B')| \\
&> |S|+|V(B')| -\left(\frac{1}{2}+\delta \right)n +\left(1-\delta\right)n-|V(B')| \\
&\geq \left(\frac{1}{2}-2\delta\right)n,
\end{align*}
a contradiction. So $B'$ contains a matching on at least $\left( \frac{1}{2} +\delta \right)n$ vertices.

We will show that $B'$ contains all vertices in $X_{R} \cup Z_{R}$. All vertices of $G$ have at most $\left(\frac{1}{4}+\delta\right)n$ non-neighbours, and so any two vertices have at least $\left(\frac{1}{2}-2\delta\right)n$ common neighbours. As $|Y_{R}| \leq \left(\frac{1}{2}-\frac{5\delta}{2}\right)n$, any pair of vertices in $Z_{R}$ have a common neighbour in $V\setminus Y_{R}$. As all vertices in $Z_{R}$ have no red neighbours in $V \setminus Y_{R}$, any two vertices in $Z_{R}$ have a common blue neighbour. Hence all vertices of $Z_{R}$ lie in the same blue component. Similarly, if $|Z_{R}| < \left(\frac{1}{2}-2\delta\right)n$ all vertices of $Y_{R}$ lie in a single blue component.

Any vertex in $X_{R}$ has at most $\left(\frac{1}{4} +\delta\right)n$ non neighbours in both $Y_{R}$ and $Z_{R}$. Thus, by \eqref{Eqn: Order of Y and Z}, every vertex in $X_{R}$ has at least one neighbour in both $Z_{R}$ and $Y_{R}$, which is necessarily blue. Hence $X_{R} \cup Z_{R}$ lies within a component of $B$ and, if $|Z_{R}| < \left(\frac{1}{2}-2\delta\right)n$, then $B$ is connected. If $B$ is not connected, then the component of $B$ containing $X_{R} \cup Z_{R}$ has order at least $n-|Y_{R}| \geq \left(\frac{1}{2}+\frac{5}{2}\delta\right)n$, and hence this component is $B'$.

Suppose now that $|V(B')| < \left(1-5\delta\right)n$ and $B'$ is bipartite, with classes $Z_{B}$ and $Y_{B}$. Both $Z_{B} \cap Z_{R}$ and $Y_{B} \cap Z_{R}$ are independent sets of $G$ and hence have order at most $\left(\frac{1}{4}+\delta \right)n$. If $|Z_{R}| < \left(\frac{1}{2}-2\delta\right)n$, then, by the above argument, $B$ is connected. So we may assume that $|Z_{R}| \geq \left(\frac{1}{2}-2\delta\right)n$. Hence, as $Z_{R} \subseteq B'=Y_{B} \cup Z_{B}$, both $Z_{B} \cap Z_{R}$ and $Y_{B} \cap Z_{R}$ have order at least $\left(\frac{1}{4}-3\delta\right)n$.

Let $v\in X_{R}$. As $X_{R} \cap X_{B} =\emptyset$, we see that $v \in Z_{B} \cup Y_{B}$. We may assume without loss of generality that $v \in Z_{B}$. Then $v$ has no blue neighbours in $Z_{B}$, and no red neighbours in $Z_{R}$. In particular, $v$ has no neighbours of either colour in $Z_{B} \cap Z_{R}$, which is a set of order at least $\left(\frac{1}{4}-3\delta\right)n$. As $v$ has at most $\left(\frac{1}{4} +\delta\right)n$ non-neighbours, it thus has at most $4\delta n$ non-neighbours in $Y_{R} \subseteq V \setminus\left(Z_{R} \cap Z_{B}\right)$. However, all edges from $v$ to $Y_{R}$ are blue. Thus all but at most $4 \delta n$ vertices in $Y_{R}$ lie in the same blue component as $v$. However, $v \in V(B')$, and $X_{R} \cup Z_{R} \subseteq B'$. Hence $B'$ contains all but $4 \delta n $ vertices, contradicting our assumption that $|V(B')| < \left(1-5\delta\right)n$. Hence, either $|V(B')| \geq \left(1-5\delta\right)n$ or $B'$ contains an odd cycle, and we are done.
\end{proof}

We shall now prove Lemma \ref{Lemma: no large independent sets}, Lemma \ref{Lemma: large sparse set} and Lemma \ref{Lemma: pathological case}, which deal with particular cases arising from the reduced graph. In both of these lemmas, we shall be using the graph $G'\subseteq G$ defined by the Regularity Lemma.

\begin{proof}[Proof of Lemma \ref{Lemma: no large independent sets}]
Suppose that $B_{G}$ has an independent set $S$ with $|S|\geq \frac{1}{2}n$. All vertices in $S$ have at most $\frac{1}{4}n$ non-neighbours in $G$, and so $\delta(R_{G}[S]) \geq |S| -\frac{1}{4}n \geq \frac{1}{2}|S|$. Hence, by Theorem \ref{Thm: Dirac}, $R_{G}[S]$ is hamiltonian. However, the minimal degree condition implies that $e(R_{G}[S]) \geq \frac{1}{4}|S|^{2}$ and so, by Theorem \ref{Thm: bondy}, either $R_{G}[S]$ is pancyclic, or $R_{G}[S] \cong K_{|S|/2,|S|/2}$. In the latter case, $\delta(R_{G}[S])=\frac{1}{2}|S|$, and so $|S| =\frac{1}{2}n$. Hence, if $|S|> \frac{1}{2}n$, then $C_{\ell} \subseteq R_{G}$ for all $\ell \in [3, |S|]$.

Suppose that $B_{G}$ is bipartite with classes $S_{1}$ and $S_{2}$, chosen so that $|S_{1}|\geq |S_{2}|$. If $|S_{1}|>\frac{1}{2}n$, then $C_{\ell} \subseteq R_{G}$ for all $\ell \in [3, |S_{1}|]$ and we are done. Hence we may assume that $n$ is even and $|S_{1}| = |S_{2}| = \frac{1}{2}n$. But by the above, we must have either that $C_{\ell} \subseteq R_{G}$ for all $\ell \in [3, \frac{1}{2}n]$, or both $R_{G}[S_{1}]$ and $R_{G}[S_{2}]$ are isomorphic to $K_{n/4,n/4}$. This implies that $n$ is divisible by four. Also, both $B_{G}[S_{1}]$ and $B_{G}[S_{2}]$ are isomorphic to the empty graph and so $G\cong K_{n/4,n/4,n/4,n/4}$.

For $i \in \{ 1,2\}$, let $S_{i,1}$ and $S_{i,2}$ be the independent sets of $G$ partitioning $S_{i}$. Then if $R_{G}$ is not bipartite, without loss of generality, there are red edges between $S_{1,1}$ and both $S_{2,1}$ and $S_{2,2}$. Hence there is a red path of length either two or four between a vertex of $S_{2,1}$ and a vertex of $S_{2,2}$, with all internal vertices in $S_{1}$. As $R_{G}$ is complete between $S_{2,1}$ and $S_{2,2}$, $R_{G}$ contains $C_{\ell}$ for all $\ell \in [4, \left\lceil \frac{1}{2}n\right\rceil]$. If, however, $R_{G}$ is bipartite then the colouring is a $2$-bipartite $2$-edge colouring.
\end{proof}

\begin{proof}[Proof of Lemma \ref{Lemma: large sparse set}]
Let $S' \subseteq V(G)$ be the union of the clusters in $S$. Then $|S'|\geq \left(\frac{2}{3}-\delta\right)n$. Let $v \in S'$. The only red neighbours of $v$ in $G'[S']$ lie in clusters adjacent in $R_{H}$ to the cluster containing $v$. Hence $v$ has at most $10 \delta km \leq 10\delta n$ red neighbours in $G'[S']$. However, the Regularity Lemma implies that $d_{G'}\left(v\right) > d_{G}\left(v\right) - \left(d+\epsilon\right) n$. Hence $\Delta(R_{G}[S']) \leq 11 \delta n$.

Any vertex $v\in S'$ has at least $|S'| -\frac{1}{4}n$ neighbours in $S'$ and so
\begin{align*}
\delta(B_{G}[S']) &\geq |S'| -\left(\frac{1}{4}+11\delta\right)n \\
&> \frac{1}{2} |S'|.
\end{align*}
Thus, by Theorem \ref{Thm: Dirac} and Theorem \ref{Thm: bondy}, the graph $B_{G}[S']$ is pancyclic. In particular $C_{\ell} \subseteq B_{G}$ for all $\ell \in [3, \left(\frac{2}{3}-\delta\right)n]$.
\end{proof}

\begin{proof}[Proof of Lemma \ref{Lemma: pathological case}]
For $1\leq i \leq 4$, let $W_{i}\subseteq V$ be the union of the clusters in $U_{i}$, so that $\underset{i}{\min} |W_{i}| \geq \left(\frac{1}{4}-4\delta\right)n$. Then $V_{0}$ is the set of all remaining vertices. Note that in $G'$ there are no blue edges from $W_{1} \cup W_{3}$ to $W_{2} \cup W_{4}$ and no red edges from $W_{1} \cup W_{2}$ to $W_{3} \cup W_{4}$.

Recall that $\delta(G') \geq \left(\frac{3}{4}-\delta\right)n$ and hence vertices in $W_{1}\cup\cdots\cup W_{4}$ have at most $\left(\frac{1}{4}+\delta\right)n$ non-neighbours in $G'$. For a vertex in $W_{1}$, at least $\left(\frac{1}{4}-4\delta\right)n$ of these non-neighbours are in $W_{4}$. Hence vertices in $W_{1}$ are adjacent in $G'$ (and hence in $G$) to all but at most $5\delta n$ vertices in $W_{1} \cup W_{2} \cup W_{3}$. Similar results hold for $W_{2}$, $W_{3}$ and $W_{4}$. Hence $\delta\left(G[W_{i}]\right)\geq |W_{i}|-5\delta n$ for all $1\leq i \leq 4$. Also, the bipartite graphs $B_{G}[W_{1} ,W_{3}]$, $B_{G}[W_{2} ,W_{4}]$, $R_{G}[W_{1} ,W_{2}]$ and $R_{G}[W_{3} ,W_{4}]$ have minimal degree at least $\left(\frac{1}{4}-9\delta\right)n$.

Our main tools to prove the lemma will be the following two claims. The first excludes a particular case, while the second gives us long monochromatic paths.
\begin{claim}
\label{Claim: no small disconnecting set}
There is no set $S$ of order at most three such that $V\setminus S$ can be partitioned into non-empty sets $X_{1},\ldots, X_{4}$ such that, for $i=1,\ldots,4$ $G$ has no edges between $X_{i}$ and $X_{5-i}$.
\end{claim}
\begin{proof}
Suppose that there is such a set $S$. Then $\sum_{i=1}^{4} |X_{i}| \geq n-3$, and so, for some $1\leq i \leq 4$, we have $|X_{i}| \geq \frac{1}{4}(n-3)$. As $X_{5-i} \neq \emptyset$, we may consider a vertex $v \in X_{5-i}$. Then $v$ has no neighbours in $X_{5-i}$ and is also not adjacent to itself. Hence $d_{G} (v) \leq n- \left(|X_{i}| +1\right) < \frac{1}{4}n$, contradicting the minimal degree of $G$.
\end{proof}

\begin{claim}
\label{Claim: no cutvertex}
For any two vertices $u$ and $w$ of $W_{1} \cup W_{2}$, the graph $R_{G}[W_{1}\cup W_{2}]$ contains a $u$-$w$ path of length $\ell$ for all $\ell \in [2,\left(\frac{1}{2}-29\delta \right)n]$ of a given parity (odd if $u\in W_{1}$ and $w\in W_{2}$ or vice versa and even otherwise). If there is a red edge in $W_{1}$ or $W_{2}$, other than $uw$, then $R_{G}[W_{1}\cup W_{2}]$ contains a red $u$-$w$ path of length $\ell$ for all $\ell \in [6,\left(\frac{1}{2}-29\delta \right)n]$.

Futhermore $R_{G}[W_{1}\cup W_{2}]$ contains a cycle of length $\ell$ for all even $\ell \in [4,\left(\frac{1}{2}-29\delta \right)n]$. If there is a red edge in $W_{1}$ or $W_{2}$, then $R_{G}[W_{1}\cup W_{2}]$ contains a cycle of length $\ell$ for all $\ell \in [4,\left(\frac{1}{2}-29\delta \right)n]$.
\end{claim}
\begin{proof}
Let $r\leq \min \{|W_{1}|, |W_{2}|-10\delta n\}$. We may assume either that $u$ and $w$ both are in $W_{1}$, in which case we let $v_{1}=u$, or that $u \in W_{2}$ and $w \in W_{1}$, in which case we let $v_{1}$ be some red neighbour of $u$ in $W_{1}\setminus\{v\}$.

Let $v_{2}, \ldots v_{r-1}$ be a sequence of vertices in $W_{1}\setminus \{v_{1},w\}$. Recall that, for all $1\leq i \leq r-1$, the vertex $v_{i}$ is adjacent in $R_{G}$ to all but at most $5\delta n$ vertices of $W_{2}$. Hence, for all $1\leq i \leq r-2$, the vertices $v_{i}$ and $v_{i+1}$ have at least $|W_{2}|-10\delta n-1\geq r-1$ common red neighbours in $W_{2}\setminus\{u\}$. Hence, there are distinct vertices $w_{i}\in W_{2}\setminus\{u\}$ for $1\leq i \leq r-1$ such that $v_{i}w_{i}$ and $w_{i}v_{i+1}$ are edges of $R_{G}$. For all $1\leq i\leq r-1$, the vertices $w$ and $v_{i}$ have at least $r$ common red neighbours in $W_{2}$, and so at least one in $W_{2} \setminus\{w_{1},\ldots,w_{r-2},u\}$. Hence, for all $1\leq i\leq r-1$, there are (not necessarily distinct) vertices $u_{i}$ in $W_{2}\setminus\{w_{1},\ldots,w_{r-2},u\}$ such that $u_{i}$ is adjacent to both $w$ and $v_{i}$ in $B_{G'}$. Similarly, for all $1\leq i\leq r-1$, there are (not necessarily distinct) vertices $u'_{i}$ in $W_{2}\setminus\{w_{1},\ldots,w_{r-2},u\}$ such that $u'_{i}$ is adjacent to both $v_{1}$ and $v_{i}$ in $B_{G'}$.

Then, for all $1 \leq i \leq r-1$,the graph $R_{G'}[W_{1} \cup W_{2}\setminus \{u\}]$ contains a $v_{1}$-$w$ path
\begin{displaymath}
v_{1}w_{1}v_{2}w_{2}\ldots w_{i-1}v_{i}u_{i}w
\end{displaymath}
and, for $i\geq 2$ a cycle
\begin{displaymath}
v_{1}w_{1}v_{2}w_{2}\ldots w_{i-1}v_{i}u'_{i}v_{1},
\end{displaymath}
both of length $2i$. Note that we thus have the required even cycles and fixed parity $u$-$w$ paths.

Suppose that there is a red edge $xy \neq uw$ in $W_{1}$ or $W_{2}$. If the edge is in $W_{1}$, we choose the sequence $v_{2}, \ldots v_{r-1}$ so that one of $v_{1}v_{2}$ and $v_{2}v_{3}$ is the edge $xy$ (which depends on whether $w \in \{x,y\}$). Then, for all $3 \leq i \leq r-1$, the graph $R_{G}[W_{1} \cup W_{2}\setminus \{u\}]$ contains a $v_{1}$-$w$ path
\begin{displaymath}
v_{1}v_{2}w_{2}\ldots w_{i-1}v_{i}u_{i}w \mbox{ or } v_{1}w_{1}v_{2}v_{3}\ldots w_{i-1}v_{i}u_{i}w
\end{displaymath}
of length $2i-1$.

Note that we also have a cycle
\begin{displaymath}
v_{1}v_{2}w_{2}\ldots w_{i-1}v_{i}u'_{i}v_{1} \mbox{ or } v_{1}w_{1}v_{2}v_{3}\ldots w_{i-1}v_{i}u'_{i}v_{1}
\end{displaymath}
of length $2i-1$. The case when the edge is in $W_{2}$ is similar.
\end{proof}

Similar results hold for each of $R_{G}[W_{3}\cup W_{4}]$, $B_{G}[W_{1}\cup W_{3}]$ and $B_{G}[W_{2} \cup W_{4}]$. In particular, there is an edge of some colour in $W_{1}$ and so $G$ contains monochromatic cycles of length $\ell$ for all $\ell \in [4,\left(\frac{1}{2}-29\delta \right)n]$. To complete the proof, we need to show that $G$ contains a monochromatic cycle of length $\ell$ for all $\ell \in [\left(\frac{1}{2}-29\delta \right)n, \left\lceil \frac{1}{2}n\right\rceil]$, and a monochromatic cycle of length at least $\left(1-59\delta\right)n$.

Suppose now that there are two disjoint paths $P_{1}$ and $P_{2}$ from $W_{1}\cup W_{2}$ to $W_{3}\cup W_{4}$ in $R_{G}$. Let $P_{1}$ have endpoints $u$ in $W_{1} \cup W_{2}$ and $u'$ in $W_{3} \cup W_{4}$ and let $P_{2}$ have endpoints $w$ in $W_{1} \cup W_{2}$ and $w'$ in $W_{3} \cup W_{4}$. By restricting to a smaller path if necessary, we may assume that all internal vertices of $P_{1}$ and $P_{2}$ are in $V_{0}$. Then, if there is any red edge in $W_{1}$, other than $uw$, we may use Claim \ref{Claim: no cutvertex} to find $u$-$w$ paths of length $\ell$ for all $\ell \in [6,\left(\frac{1}{2}-29\delta \right)n]$ in $R[W_{1},W_{2}]$. However, Claim \ref{Claim: no cutvertex} also implies that $R_{G}[W_{3},W_{4}]$ contains $u'$-$w'$ paths of length $\ell$ for all $\ell \in [6,\left(\frac{1}{2}-29\delta \right)n]$ of a given parity. By concatenating these paths with $P_{1}$ and $P_{2}$ we see that in this case we have monochromatic cycles of length $\ell$ for all $\ell \in [3,\left(1-58\delta \right)n]$.

So we may assume that if there are two disjoint paths from $W_{1}\cup W_{2}$ to $W_{3}\cup W_{4}$ in $R_{G}$, then $\sum_{i=1}^{4} e(R_{G}[W_{i}]) \leq 2$. Similarly, if there are two disjoint paths from $W_{1}\cup W_{3}$ to $W_{2}\cup W_{4}$ in $B_{G}$, then $\sum_{i=1}^{4} e(B_{G}[W_{i}]) \leq 2$. However, $e(G[W_{i}]) \geq \left(\frac{1}{8}-\frac{9\delta}{2}\right)n|W_{i}|$ for all $1 \leq i \leq 4$ and so without loss of generality we may assume that there are no two disjoint red paths from $W_{1}\cup W_{2}$ to $W_{3}\cup W_{4}$ in $R_{G}$.

By a corollary of Menger's Theorem, there is a vertex $v_{R}$ such that there are no red paths from $W_{1} \cup W_{2}$ to $W_{3} \cup W_{4}$ in $G -\{v_{R}\}$. If there is also a vertex $v_{B}$ such that there are no blue paths from $W_{1} \cup W_{3}$ to $W_{2} \cup W_{4}$ in $G - \{ v_{B} \}$, then taking $S=\{v_{R},v_{B}\}$, we would have a contradiction to Claim \ref{Claim: no small disconnecting set}. Hence we may assume that there are two disjoint blue paths between $W_{1} \cup W_{3}$ and $W_{2} \cup W_{4}$. Hence, applying Claim \ref{Claim: no cutvertex} to the ends of these paths as above there is a blue cycle of length at least $\left(1-59\delta\right)n$. Thus $G$ contains a monochromatic cycle of length at least $\left(1-59\delta\right)n$.

To complete the proof, we need to show that $G$ contains a monochromatic cycle of length $\ell$ for all $\ell \in [\left(\frac{1}{2}-29\delta \right)n, \left\lceil \frac{1}{2}n\right\rceil]$. We shall find a lower bound on the red degree of each vertex. Recall that we are assuming that there are two disjoint blue paths between $W_{1} \cup W_{3}$ and $W_{2} \cup W_{4}$. Hence we may assume that $e(B_{G}[W_{1}]) \leq 1$, or else we have blue cycles of length $\ell$ for all $\ell \in [3,\left(1-58\delta \right)n]$ as above. If $v \in W_{1}$, then in $G'$, $v$ has at most $5\delta n$ non-neighbours in $W_{1} \cup W_{2}$ and no blue neighbours in $W_{2}$. However $v$ has at most one blue neighbour in $W_{1}$ in $G$ and hence in $G'$. Thus all vertices in $W_{1}$ have red degree at least $|W_{1}|+|W_{2}| -5\delta n -1$ in $G'$ and hence in $G$. Similar bounds hold for all vertices in $\bigcup_{i=1}^{4}W_{i}$.

Suppose that some vertex $v \in V_{0}$ has at least $\left(\frac{1}{2}+8\delta\right)n +3$ blue neighbours. Then it must have at least two blue neighbours in at least three of the sets $W_{i}$. Suppose that there is a blue path $P$ from $W_{1} \cup W_{3}$ to $W_{2} \cup W_{4}$ in $G-\{v\}$. Without loss of generality, we may assume that $P$ has endpoints $u' \in W_{1}$ and $u \in W_{2}$ and all internal vertices of $P$ are in $V_{0}$. Suppose that $v$ has at least two blue neighbours in each of $W_{1}$, $W_{2}$ and $W_{3}$, the other cases being similar. We may find $w' \in W_{1}$, $w \in W_{2}$ and $w'' \in W_{3}$ with $\{u,u'\} \cap \{ w,w',w''\}=\emptyset$ such that each of $w$, $w'$ and $w''$ are blue neighbours of $v$.

By Claim \ref{Claim: no cutvertex} we have the following paths:
\begin{itemize}
\item for all even $\ell \in [6,\left(\frac{1}{2}-29\delta \right)n]$, $B[W_{2},W_{4}]$ contains a $u$-$w$ path $P_{\ell}$ of length $\ell$;
\item for all even $\ell' \in [6,\left(\frac{1}{2}-29\delta \right)n]$, $B[W_{1},W_{3}]$ contains a $u'$-$w'$ path $P'_{\ell'}$ of length $\ell'$;
\item for all odd $\ell'' \in [7,\left(\frac{1}{2}-29\delta \right)n]$, $B[W_{1},W_{3}]$ contains a $u$-$w''$ path $P''_{\ell''}$ of length $\ell''$.
\end{itemize}
Then, for all even $\ell,\ell' \in  [6,\left(\frac{1}{2}-29\delta \right)n]$, the path
\begin{displaymath}
uP_{\ell}wvw'P'_{\ell'}u'
\end{displaymath}
is a blue $u$-$u'$ path of length $2+\ell +\ell'$ which is internally disjoint from $P$. Similarly, for all even $\ell \in  [6,\left(\frac{1}{2}-29\delta \right)n]$ and odd $\ell'' \in  [6,\left(\frac{1}{2}-29\delta \right)n]$, the path
\begin{displaymath}
uP_{\ell}wvw''P''_{\ell''}u'
\end{displaymath}
is a blue $u$-$u'$ path of length $2+\ell +\ell''$ which is internally disjoint from $P$.

Hence, for all $L \in [14, \left(1-58\delta\right)n]$, there is a blue $u$-$u'$ path of length $L$ which is internally disjoint from $P$. Since $|P|\leq |V_{0}|+2\leq \delta n$, this gives blue cycles of length $L$ for all $L \in[\delta n +14,\left(1-58\delta\right)n]$. As we have already shown that $G$ contains monochromatic cycles of length $\ell$ for all $\ell \in [3,\left(\frac{1}{2}-29\delta \right)n]$, we are done. Hence, if there is a vertex $v \in V_{0}$ with blue degree at least $\left(\frac{1}{2}+8\delta\right)n +3$, there are no blue paths from $W_{1} \cup W_{3}$ in $W_{2} \cup W_{4}$ in $G-\,v\}$. This contradicts Claim \ref{Claim: no small disconnecting set}, with $S=\{v,v_{R}\}$. Thus each vertex in $V_{0}$ has blue degree at most $\left(\frac{1}{2}+8\delta\right)n +3$, and so red degree at least $\left(\frac{1}{4}-8\delta\right)n-3$.

Let $C_{1}$ be the red component of $G-\{v_{R}\}$ containing $W_{1} \cup W_{2}$ and $C_{2}$ be the red component of $G-\{v_{R}\}$ containing $W_{3} \cup W_{4}$. We know that $R_{G}[W_{1} \cup W_{2}]$ and $R_{G}[W_{3}\cup W_{4}]$ are connected, and the minimal red degree condition in $V_{0}$ ensures that there are at most two components in $R_{G}[V-\{v_{R}\}]$. As $v_{R}$ has red degree at least $\left(\frac{1}{4}-8\delta\right)n-3$, it has at least $\left(\frac{1}{8}-5\delta\right)n$ red neighbours in at least one of $C_{1}$ or $C_{2}$. Let $C'_{i}$ be the set $C_{i}$, with $v_{R}$ added if it has at least $\left(\frac{1}{8}-5\delta\right)n$ red neighbours in $C_{i}$.

Then $|C'_{1}|+|C_{2}'|\geq n$ and so we may assume without loss of generality that $|C_{1}'|\geq \left\lceil \frac{1}{2}n\right\rceil$. All vertices in $C_{1}'$ have degree in $R[C_{1}']$ at least $\left(\frac{1}{8}-5\delta\right)n$. Further, all vertices in $C_{1}'\setminus|V_{0}|$ have degree in $R[C_{1}']$ at least $|C_{1}'|-6\delta n$. As $|C_{1}'| \leq \left(\frac{1}{2}+8\delta \right)n$ and $|V_{0}|\leq \epsilon n$, the condition of Theorem \ref{Thm: chvatal} holds on $R[C_{1}']$ and so $R[C_{1}']$ is hamiltonian. But we also have 
\begin{align*}
e(R[C_{1}'])&\geq \frac{1}{2}\left(|C_{1}'|-6\delta n\right) \left( |C_{1}'|-|V_{0}|\right) \\
&> \frac{1}{4} |C_{1}'|^{2}.
\end{align*}
Hence, by Theorem \ref{Thm: bondy}, $R[C_{1}']$ is pancyclic and we are done.
\end{proof}

\section{Monochromatic circumference}
\label{Sec: Monochromatic circumference}
In this section we shall look at the monochromatic circumference of a graph. We begin by proving Theorem \ref{Thm: circumference}.
\begin{proof}[Proof of Theorem \ref{Thm: circumference}]
As in the proof of Theorem \ref{Thm: main result}, we consider the reduced graph $H$, which has order $k$ and minimal degree at least $\left(\frac{3}{4}-\delta\right)k$. Applying Lemma \ref{Lemma: reduced graph}, we have one of the following.
\begin{enumerate}[(i)]
\item There is a component of $R_{H}$ or $B_{H}$ which contains a matching on at least $\left( \frac{2}{3} +\delta \right)k$ vertices.
\item There is a set $S$ of order at least $\left(\frac{2}{3}-\frac{\delta}{2}\right)k$ such that either $\Delta(R_{H}[S]) \leq 10\delta k$ or $\Delta(B_{H}[S]) \leq 10\delta k$.
\item There is a partition $V(H)=U_{1} \cup \cdots \cup U_{4}$ with $\underset{i}{\min} |U_{i}| \geq \left(\frac{1}{4}-3\delta\right)k$ such that there are no blue edges from $U_{1} \cup U_{2}$ to $U_{3} \cup U_{4}$ and no red edges from $U_{1} \cup U_{3}$ to $U_{2} \cup U_{4}$.
\end{enumerate}

In the first case, we use the Blow-Up Lemma as in Theorem \ref{Thm: main result} to find a monochromatic cycle of length at least $\left(\frac{2}{3}+\frac{\delta}{2}\right)n$. In the second case, assume without loss of generality that $\Delta(R_{H}[S]) \leq 10\delta k$. Then, by Lemma \ref{Lemma: large sparse set}, $G$ contains a blue cycle of length $\ell$ for all $\ell \in \left[3,\left(\frac{2}{3}-\delta\right)n\right]$. In the third case \ref{Lemma: pathological case} implies that $G$ contains a monochromatic cycle of length at least $\left(1-59\delta\right)n \geq \left(\frac{2}{3}+\delta\right)n$.
\end{proof}

We will make the following definition.
\begin{defn}
For $0<c<1$, let $\Phi=\Phi_{c}$ be the supremum of values $\phi$ such that any graph $G$ of sufficiently large order $n$ with $\delta(G)>cn$ and a $2$-colouring $E(G)=E(R) \cup E(B)$ has monochromatic circumference at least $\phi n$.
\end{defn}

For $c\geq \frac{3}{4}$, Theorem \ref{Thm: circumference} implies that $\Phi_{c} \geq \frac{2}{3}$. However, the example given after Theorem \ref{Thm: circumference} shows that $\Phi_{c} \leq \frac{2}{3}$ for all $c$. We can also find upper and lower bounds for $\Phi_{c}$ when $c<\frac{3}{4}$, and we collect them into the following theorem.

\begin{thm}
\label{Thm: upper and lower bounds for circumference}
For all $c\geq \frac{3}{4}$, we have $\Phi_{c}=\frac{2}{3}$. For all $c\in(0,1)$, we have $ \Phi_{c} \geq  \frac{1}{2}c$. Also, there are the following upper bounds on $\Phi_{c}$.
\begin{align*}
\Phi_{c} \leq  
\begin{cases}
\frac{1}{2} &\quad c \in [\frac{3}{5},\frac{3}{4}) \\
\frac{2}{5} &\quad c \in [\frac{5}{9},\frac{3}{5}) \\
\frac{1}{r} &\quad c < \frac{2r-1}{r^{2}} \mbox{ for all } r\geq 3.
\end{cases}
\end{align*}
\end{thm}

Note that, as $c \to 0$, we may use the last upper bound to show that $\frac{\Phi_{c}}{\frac{1}{2}c} \to 1$. Hence, asymptotically, as $c \to 0$, the upper and lower bounds on $\Phi_{c}$ agree.

\begin{proof}[Proof of Theorem \ref{Thm: upper and lower bounds for circumference}]
For $c \in (0,1)$, a $2$-edge coloured graph with $\delta(G)>cn$ has at least $\frac{c}{2}n^{2}$ edges. Hence there are at least $\frac{c}{4}n^{2}$ edges of one colour. We may deduce from Theorem \ref{Thm: erdos-gallai} that, in that colour, there is a cycle of length at least $\frac{1}{2}c$. Hence $\Phi_{c} \geq \frac{1}{2} c$ for all $c \in (0,1)$. We now prove the upper bounds on $\Phi_{c}$.

For $c \in [\frac{5}{9},\frac{3}{5})$, let $t$ be an integer such that $t> \frac{1}{3-5c}$. We define a graph $G_{t}'$ as follows. Let $S_{1}$ and $S_{2}$ be sets of order $2t$ and $T$ be a set of order $t$. Let $R$ be the union of the complete graph on $S_{1}$ and the complete graph on $S_{2}$. Then $R$ has circumference $2t$. Let $B$ be the union of the complete graph on $T$ and the complete bipartite graph between $T$ and $S_{1} \cup S_{2}$. Then, any two consecutive vertices of a cycle in $B$ must contain a vertex of $T$ and hence $B$ has circumference at most $2t$. Let $G'_{t}$ be the union of $R$ and $B$. Then $\delta(G) =3t-1 > c|G'_{t}|$ and so $\Phi_{c} \leq \frac{2}{5}$.

\begin{figure}[h!]
\centering
\includegraphics{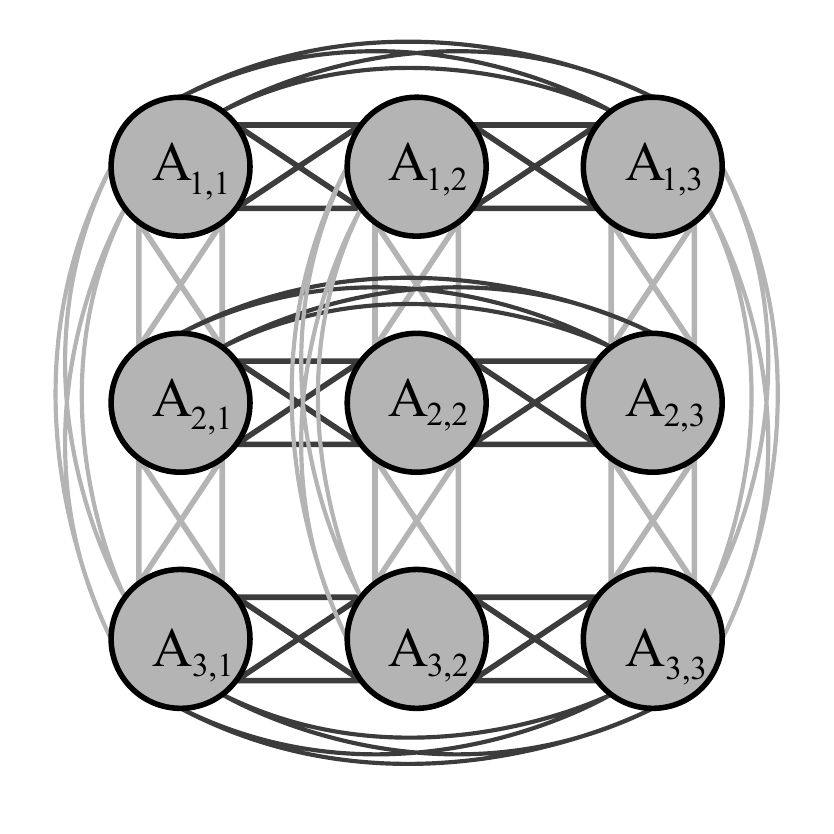}
\caption{The graph $G^{(3)}_{t}$}
\label{Fig: 1/2<c<3/4}
\end{figure}

If, for $r\geq 2$, we have $c \in(0, \frac{2r-1}{r^{2}})$, let $t$ be an integer such that $t>\frac{1}{2r-1-R^{2}c}$. Define a family $\{A_{i,j}:1\leq i \leq r, 1\leq j \leq r\}$ of sets of order $t$. We define the following graphs on vertex set $\bigcup_{i,j} A_{i,j}$:
\begin{align*}
E(B)&= \{ uv: u \in A_{i,j} , v\in A_{i,j'} \mbox{ for some $1\leq i \leq r$ and $j \neq j'$} \}; \\
E(R)&= \{ uv: u \in A_{i,j} , v\in A_{i',j} \mbox{ for some $1\leq j \leq r$.} \}
\end{align*}
Let $G^{(r)}_{t}$ be the union of the graphs $R$ and $B$, as illustrated in Figure \ref{Fig: 1/2<c<3/4}, for the case $r=3$. Then $|G^{(r)}_{t}|=r^{2}t$ and $\delta(G''_{r,t}) = (2r-1)t-1>c|G^{(r)}_{t}|$. However as all monochromatic components have order $rt$, there are no monochromatic cycles of length greater than $\frac{1}{r}\left|G^{(r)}_{t}\right|$. Hence $\Phi_{c} \leq \frac{1}{r}$. Note that this case includes the bound $\Phi_{c} \leq \frac{1}{2}$ for $c < \frac{3}{4}$.
\end{proof}

\section{Conclusion}
\label{Sec: conclusion}
Theorem \ref{Thm: main result} is a $2$-colour version of the uncoloured (or $1$-coloured) result of Bondy that graphs with order $n$ and minimal degree at least $\frac{1}{2}n$ are pancyclic. We may hope to generalise to $k$ colours. In this case, we let $E(G)=\bigcup_{i=1}^{k} E(G_{i})$ be an edge colouring, where each $G_{i}$ is a spanning subgraph of $G$, representing the edges coloured $i$. Our extremal graph was found by letting both $R$ and $B$ be subgraphs of the extremal graph in the uncoloured case, and we again use this method to find $k$-coloured graphs with high minimum degree but no odd cycles.
\begin{defn}
Let $n=2^{k}p$ and let $G$ be isomorphic to the $2^{k}$-partite graph with classes all of order $p$. A \emph{$k$-bipartite $k$-edge colouring} of $G$ is a $k$-edge colouring $E(G)=\bigcup_{i=1}^{k} E(G_{i})$ such that each $G_{i}$ is bipartite.
\end{defn}
As in the $2$-coloured case, we can deduce that a $k$-bipartite $k$-edge colouring of the $2^{k}$-partite graph with classes all of order $p$ induces a labelling $U_{\alpha}$ ($\alpha \in \{1 ,2\}^{k}$) of the classes such that, for all $i$, the graph $G_{i}$ is bipartite with classes
\begin{displaymath}
\bigcup_{\alpha: \alpha_{i}=1} U_{\alpha}
\end{displaymath}
and
\begin{displaymath}
\bigcup_{\alpha: \alpha_{i}=2} U_{\alpha}.
\end{displaymath}
Note that this implies that, if $\alpha$ and $\beta$ in $\{1,2\}^{k}$ differ only in the $i$th place, then all edges between $U_{\alpha}$ and $U_{\beta}$ are coloured with $i$. As this graph has minimum degree $\left(1-\frac{1}{2^{k}}\right)n$, we make the following conjecture.

\begin{conj}
Let $n\geq 3$, and $k$ be an integer. Let $G$ be a graph of order $n$ with $\delta(G)\geq \left(1-\frac{1}{2^{k}}\right)n$. If $E(G)=\cup_{i=1}^{k} E(G_{i})$ is a $k$-edge colouring, then either:
\begin{itemize}
\item  for all $\ell \in \left[\min\{2^{k},3\},\left\lceil\frac{1}{2^{k-1}}n\right\rceil\right]$ there is some $1\leq i \leq k$ such that $C_{\ell} \subseteq G_{i}$, or;
\item $n=2^{k} p$, $G$ is the complete $2^{k}$-partite graph with classes of order $p$, and the colouring is a $k$-biparitite $k$-edge colouring.
\end{itemize}
\end{conj}

Note that the case when $k=1$ is Bondy's Theorem, and the case $k=2$ is Theorem \ref{Thm: main result}.

We pose the following problem about the monochromatic circumference.

\begin{prob}
What is the value of $\Phi_{c}$ for $c<\frac{3}{4}$?
\end{prob}

Note that Theorem \ref{Thm: upper and lower bounds for circumference} shows that $\Phi_{c}=\frac{2}{3}$ for all $c \geq \frac{3}{4}$. In this case, we make the following conjecture with an exact bound on the monochromatic circumference.

\begin{conj}
Let $G$ be a graph of order $n$ with $\delta(G) \geq \frac{3}{4}n$. Let $n=3t+r$, where $r \in\{0,1,2\}$. If $E(G)=E(R_{G}) \cup E(B_{G})$ is a $2$-edge colouring, then $G$ has monochromatic circumference at least $2t+r$.
\end{conj}
Note that Theorem \ref{Thm: circumference} is an asymptotic version of this conjecture. By considering the graph $F_{2t+r,t}$ as defined in Section \ref{Sec: Ramsey intro}, we see that this conjecture is best possible.

\bibliographystyle{acm}
\bibliography{/homes/homes0/v1/pg_2007/white/Documents/Bibs/graphtheory.bib}
\end{document}